\documentclass[12pt]{amsart}
\usepackage{amsthm}
\usepackage{amssymb}
\usepackage{amsmath}
\usepackage{mathabx}
\numberwithin{equation}{section}
\newcommand{\bea}{\begin{eqnarray}}
\newcommand{\eea}{\end{eqnarray}}
\newcommand{\be}{\begin{eqnarray*}}
\newcommand{\ee}{\end{eqnarray*}}
\newtheorem{theorem}{Theorem}[section]
\newtheorem{lemma}{Lemma}[section]
\newtheorem{corollary}{Corollary}[section]

\newtheorem{example}{Example}[section]
\begin{document}
%%%%%%%%%%%%%%
\title[Isomorphism Classes of Idempotent Evolution Algebras]{Isomorphism Classes of \\ Idempotent Evolution Algebras}
\author[Yangjiang Wei]{Yangjiang Wei}
\address{School of Mathematics and Statistics, Nanning Normal University, Nanning 530100, P. R. China}
\email{gus02@163.com}
\author[Yi Ming Zou]{Yi Ming Zou$^{\ast}$}
\address{Department of Mathematical Sciences, University of Wisconsin, Milwaukee, WI 53201, USA} \email{ymzou@uwm.edu}
\thanks{Keywords: Idempotent evolution algebra, finite field, isomorphism class, counting formula}
\thanks{MSC2010: primary 17D92, secondary 05A05}
\thanks{* Corresponding author. Email: ymzou@uwm.edu}
\maketitle
%%%%%%%%%%%%%%
\begin{abstract}
We showed that isomorphism classes of idempotent evolution algebras are in bijection with the orbits of the semidirect product group of the symmetric group and the torus, considered the combinatoric problem of enumeration of isomorphism classes for these algebras over arbitrary finite fields, derived a general counting formula, and obtained explicit formulas for the numbers of isomorphism classes in dimensions $2$, $3$, and $4$ over any finite field.
\end{abstract}
\par
\medskip
%%%%%%%%%%%%%
\section{Introduction}
Evolution algebras are non-associative and commutative algebras motivated by the evolution laws of genetics \cite{Tian1}, they have applications in many other areas and have been gaining more attentions recently. In particular, the classification problem of evolution algebras has been considered in dimensions $2$, $3$, and $4$ \cite{Beh, Cas1, Cas2, Cas3, Cas4}.  From these results, it is clear that in general, such a classification is complicated: already in the case of dimensions $3$ and $4$, these results provide long lists of evolution algebras over certain fields even with incomplete classifications. So far, it is unclear what would be a good way to classify these algebras (compare with the comment in \cite[p.~187]{Art1} on the list of crystallographic groups). Here we consider the isomorphism problem of a special class of finite dimensional evolution algebras that we call {\it idempotent evolution algebras}, and considered the combinatoric problem of enumeration of isomorphism classes for these algebras over arbitrary finite fields. These finite dimensional evolution algebras have been shown to process many interesting properties that related to other topics, such as combinatoric, group theory, and matrices. In addition, for each dimension $n$, the idempotent evolution algebras form a dense subset of the set of $n$-dimensional evolution algebras \cite{Eld1, Eld2, Sri1}. We recall some relevant definitions and results.
\par
An $n$-dimensional evolution algebra $\mathcal E$ over an arbitrary field $\mathbb F$ can be defined by using a natural basis $e_1,...,e_n$ of $\mathcal E$ as a vector space over $\mathbb F$, and a structure matrix $A =(a_{ij}),\;a_{ij}\in \mathbb{F}, \;1\le i, j\le n$, such that the following conditions hold \cite[p.~20]{Tian1}:
%\bea\label{e11}
$e_ie_j = 0$, if $i\ne j$, and $(e_1^2, ..., e_n^2) = (e_1, ..., e_n)A$.
%\eea
%\par
We denote by $\mathcal{E}(A)$ the evolution algebra with the structure matrix $A$ if we need to specify $A$. The commutativity follows from the definition, but no associativity is assumed.
We call an evolution algebra $\mathcal E$ {\it idempotent} if $\mathcal{E}^2 = \mathcal E$. 
\par
It was shown in \cite[Prop.~4.2]{Eld1} that
\bea\label{e12}
\qquad\quad\mathcal{E}^2 = \mathcal E\Leftrightarrow e_1^2, ..., e_n^2\; \mbox{form a basis of}\; \mathcal E\Leftrightarrow A\; \mbox{is nonsingular}; 
\eea
and shown in \cite[Thm.~4.8]{Eld1} that the automorphism group of a finite
dimensional idempotent evolution algebra over an arbitrary field is finite. In \cite[Thm.~3.2]{Eld2}, the automorphism groups of these finite dimensional evolution algebras were described in terms of an associated graph. This allows to prove
that every finite group can be represented as the full automorphism group of an
idempotent evolution over a field of characteristic 0 \cite[Thm.~3.1]{Sri1}, or over
an arbitrary field \cite[Thm.~1.1]{Cos1}. Moreover, it was shown \cite[Thm.~A]{Cos2} that this evolution algebra can be
chosen simple.  
\par
 Let $\mathcal{E}$ (resp. $\mathcal{E}'$) be an idempotent evolution algebra with a natural basis $e_1, ..., e_n$ (resp. $e_1', ..., e_n'$) and the structure matrix $A =(a_{ij})$ (resp. $B = (b_{ij})$). For a matrix $M=(m_{ij})$, let $M^{(2)}$ denote the matrix $(m_{ij}^2)$. Recall the following \cite[Thm.~4.4]{Eld1}:
 \begin{theorem}[Elduque and Labra] \label{L1}
Two idempotent evolution algebras $\mathcal{E}(A)$ and $\mathcal{E}(B)$ are isomorphic if and only if there exists a permutation $\sigma\in S_n$, where $S_n$ is the symmetry group on $n$ objects, and an $n\times n$ invertible matrix $P=(p_{ij})$, such that $p_{ij}\ne 0 \Leftrightarrow i=\sigma(j)$ and $P^{-1}BP^{(2)} = A$.
\end{theorem}
\par
Note that the matrix $P$ in the above theorem is the product of a permutation matrix and a nonsingular diagonal matrix. To simplify our notation, we will also write the $n\times n$ permutation matrix corresponding to a permutation $\sigma\in S_n$ defined by
\bea\label{e13}
P_{\sigma}\colon \;\; (e_1, ..., e_n) \longrightarrow (e_{\sigma(1)}, ..., e_{\sigma(n)}) = (e_1, ..., e_n)P_{\sigma} 
\eea
as $\sigma$. Then for any matrix unit $e_{ij}$ and any $n\times n$ matrix $A = (a_{ij})$, we have
\be
\sigma^{-1}e_{ij}\sigma = e_{\sigma^{-1}(i)\sigma^{-1}(j)}\quad\mbox{and}\quad \sigma^{-1}A\sigma = (a_{\sigma(i)\sigma(j)}),
\ee
that is, the $ij$-entry of $A$ is $a_{\sigma(i)\sigma(j)}$. 
\par\medskip
Let $GL_n(\mathbb{F})$ be the set of all nonsingular $n\times n$ matrices over $\mathbb{F}$, and let $T_n\subset GL_n(\mathbb{F})$ be the subset of the nonsingular diagonal matrices. For an element $t = \mbox{diag}(t_1, ..., t_n)\in T_n$, we can simply write $t = (t_1, ..., t_n)$ if there is no confusion. Let 
\be
G=S_n\ltimes T_n =\{\sigma t\; |\;\sigma\in S_n, t\in T_n\}
\ee 
be the semidirect product of the multiplicative groups $S_n$ and $T_n$ with the action of $S_n$ on $T_n$ given as follows. For $\sigma \in S_n$ and $t = (t_1, ..., t_n)\in T_n$,
\bea\label{e14}
 \sigma^{-1} (t_1, ..., t_n) \sigma = (t_{\sigma(1)}, ..., t_{\sigma(n)}) =: \sigma(t).
\eea
Then $G$ is a subgroup of $GL_n(\mathbb{F})$. Since the elements of $GL_n(\mathbb{F})$ are the defining (structure) matrices of the idempotent evolution algebras over $\mathbb{F}$, we are interested in the action of $G$ on $GL_n(\mathbb{F})$ given by Theorem \ref{L1} detailed below. 
\par
 Note that if $\sigma\in S_n, t\in T_n$, then since $t$ is a diagonal matrix and $\sigma$ is a $(0,1)$-matrix, we have
\bea\label{e15}
(\sigma t)^{(2)} =\sigma t^{(2)} = \sigma t^2,
\eea
where the $t^2$ at the end is the usual square of the diagonal matrix $t$. The following action of $G$ on $GL_n(\mathbb{F})$ is a generalization of the action defined in \cite[Eqn.~(12)]{Cas2}:
{\small \bea\label{e16}
\qquad \sigma t\colon A\mapsto (\sigma t)^{-1}A(\sigma t)^{(2)} = t^{-1}\sigma^{-1} A \sigma t^2, \; \sigma t \in G,\; A\in GL_n(\mathbb{F}).
\eea}
Then we have the following corollary of Theorem \ref{L1}, which is a generalization of  \cite[Prop.~3.2(iv)]{Cas2}:   
\begin{corollary}\label{C1} The isomorphism classes of $n$-dimensional idempotent evolution algebras over an arbitrary field $\mathbb F$ are in one-to-one correspondence with the orbits of the $G$-action on $GL_n(\mathbb{F})$ defined by (\ref{e16}).
\end{corollary}
\par\medskip
In the rest of this paper, we consider the combinatoric enumeration problem of these $G$-orbits over an arbitrary finite field. In section 2, we derive a counting formula for these orbits via Burnside's lemma. However, in order to actually enumerate the isomorphism classes, further computation method must be developed. Thus in section 3, we develop an approach to compute the numbers of these orbits by applying matrix theory. In this approach, certain determinants of order $n$ need to be considered, which leads to high order multivariable polynomial equations.  In this paper, we are able to derive explicit formulas for the numbers of isomorphism classes of idempotent evolution algebras in dimensions $2$, $3$, and $4$ over any finite field. These will be presented in sections $4$, $5$, and $6$, respectively.
\par
We should point out that, in the classification work of \cite{Beh, Cas2, Cas4} on evolution algebras of dimensions $3$ and $4$, it was assumed that these evolution algebras are over a field $\mathbb{F}$ of characteristic different from $2$ and in which every polynomial of the form $x^n-a$, for $n = 2,3,7$ and $a\in \mathbb{F}$, has a root in the field. These conditions exclude all finite fields: it is well-known that the multiplicative group of a finite field $\mathbb{F}_q$ of $q$ elements is a cyclic group \cite[Thm.~15.7.3]{Art1} of order $q-1$, so if the characteristic of the field is an odd prime, then every equation of the form $x^2-\xi=0$, where $\xi$ is a generator of the multiplicative group, has no solution in $\mathbb{F}_q$.  
\par\medskip
%%%%%%%%%%%%%%%%%%%%%%%%%%%
\section{A counting formula}
\par\medskip 
From now on, we assume that $\mathbb{F}_q$ is a finite field of $q$ elements, where $q = p^m$ for some prime integer $p$ and some integer $m > 0$. The $n$-dimensional idempotent evolution algebras over $\mathbb{F}_q$ are defined by the elements $A$ of $GL_n(\mathbb{F}_q)$ and the corresponding natural bases $(e_1, ..., e_n)$. The case $n = 1$ is trivial, so we assume that $n>1$. For the action of $G=S_n\ltimes T_n$ on $X := GL_n(\mathbb{F}_q)$ given by (\ref{e16}), Burnside's lemma says that the number of $G$-orbits in $X$ can be computed by the following formula:
\bea\label{e21}
|X/G| = \frac{1}{|G|}\sum_{x\in X}|G_x|=\frac{1}{|G|}\sum_{g\in G}|X^g|,
\eea
where $G_x = \{g\in G\;|\; gx=x\}$ is the stabilizer of $G$ at $x$, and $X^g = \{x\in X\;|\; gx=x\}$ is the set of fixed points of $g$.  Since $|G| = |S_n|\cdot |T_n| =n!(q-1)^n$ and  
\bea\label{e22}
 |X| = \prod_{i=0}^{n-1}(q^n - q^i) = q^{n(n-1)/2}\prod_{i=1}^n(q^i - 1),
\eea
except for some small cases, the number of summands in the second sum in (2.1) is much smaller than that of the first sum. We will use the group structure of $G$ to further reduce the number of terms in the summation.  The following lemma holds over any field $\mathbb{F}$.
\begin{lemma}\label{L2}
For $\sigma, \tau \in S_n$ and $t\in T_n$, $A\in X$ is a fixed point of $\sigma t$ if and only if $\tau^{-1}A\tau$ is a fixed point of $\tau^{-1}(\sigma t)\tau = (\tau^{-1}\sigma\tau)\tau(t)$.
\end{lemma}
\begin{proof} 
This follows from a basic fact in group actions: if $x$ is a fixed point of $g$, i.e. $g\cdot x = x$, then $h\cdot x$ is a fixed point of $hgh^{-1}$. Notice that the last equality holds due to (\ref{e14}). 
\end{proof}
\par
We recall some basic facts about the symmetric group (see for example \cite{Sag1}). Let $\mathcal{C}(S_n)$ be the set of all conjugacy classes of $S_n$, and let $\mathcal{P}(n)$ be the set of all partitions of $n$. Then there is a one-to-one correspondence between $\mathcal{C}(S_n)$ and $\mathcal{P}(n)$. For our purpose, we use the following notation for a partition of $n$: 
\bea\label{e23}
\mu = \{0<\mu_1\le \mu_2\le \cdots \le \mu_r \},\;\mbox{such that}\; \sum_{i=1}^r\mu_i = n.
\eea
This is the reverse version of the more commonly used notation, since we want to place the shorter cycles of a permutation at the beginning. This will become clear later. 
\par
For a partition $\mu$ of $n$, we will also use $\mu$ to denote the permutation defined by partitioning $(1,2, ..., n)$ into disjoint cycles according to $\mu$. For example, if $n=7$ and $\mu =\{0<1\le 1<2 < 3\}$, then as a permutation $\mu = (1)(2)(34)(567)=(34)(567)$. For $\sigma\in S_n$, we denote by $m_k(\sigma), 1\le k\le n,$ the number of cycles of length $k$ in $\sigma$ when $\sigma$ is written as a disjoint product of cycles.  For the $\mu$ just mentioned, $m_1(\mu) = 2, m_2(\mu)= m_3(\mu) = 1$, and $m_k(\mu) = 0, k>3$.
\par
For $\mu\in\mathcal{P}(n)$, let $C(\mu)\in\mathcal{C}(S_n)$ be the conjugacy class defined by $\mu$ and let $c(\mu) = |C(\mu)|$ be the number of elements in $C(\mu)$. Let 
\be
d(\mu):=\prod_{k=1}^nm_k(\mu)!k^{m_k(\mu)}.
\ee 
Then we have \cite[page 3, Eqn. (1.2)]{Sag1}:
\bea\label{e24}
c(\mu) = \frac{n!}{d(\mu)}.
\eea
For $t\in T_n$, let $b(\mu,t) = |X^{\mu t}|$ be the number of fixed points of the group element $\mu t\in G = S_n\ltimes T_n$, and let 
\bea\label{e25}
B(\mu) := \frac{\sum_{t\in T_n}b(\mu, t)}{d(\mu)}.
\eea
\begin{theorem}\label{T1} 
For $n\ge 1$, the number of isomorphism classes of idempotent evolution algebras over $\mathbb{F}_q$ is given by
\bea\label{e26}
\mathcal{N}(n,\mathbb{F}_q) := |X/G| = \frac{\sum_{\mu\in \mathcal{P}(n)}B(\mu)}{(q-1)^n}.
\eea
\end{theorem}
\begin{proof}
We first remark that for the trivial case $n=1$, the right side of the formula is $1$. Then assume that $n>1$. By Lemma 2.1, if $\sigma, \sigma'\in S_n$ are conjugates, say $\sigma' = \tau^{-1}\sigma\tau$, then $|X^{\sigma t}| = |X^{\sigma' \tau(t)}|$. Thus though $|X^{\sigma t}|$ may not be equal to  $|X^{\sigma' t}|$ for the same $t\in T_n$, since $\tau(T_n) = T_n$  (with the action of $\tau$ given by (\ref{e14})), we have:
\bea\label{e27}
\sum_{t\in T_n}|X^{\sigma t}| = \sum_{t\in T_n}|X^{\sigma' t}|.
\eea
Therefore, by Burnside's lemma (see (2.1)),
\be
|X/G| &=& \frac{1}{|G|}\sum_{g\in G}|X^g| 
= \frac{1}{|G|}\sum_{\mu\in \mathcal{P}(n)}\sum_{\sigma\in C(\mu)}\sum_{t\in T_n}|X^{\sigma t}|\\
{} &=& \frac{1}{(q-1)^nn!}\sum_{\mu\in \mathcal{P}(n)}c(\mu)\sum_{t\in T_n}|X^{\mu t}|\quad \mbox{(by (2.7))}\\
{} &=& \frac{1}{(q-1)^nn!}\sum_{\mu\in \mathcal{P}(n)}\frac{n!}{d(\mu)}\sum_{t\in T_n}b(\mu, t)\;\mbox{(by (2.4))}\\
{} &=& \frac{1}{(q-1)^n}\sum_{\mu\in \mathcal{P}(n)}\frac{\sum_{t\in T_n}b(\mu, t)}{d(\mu)}\\
{} &=&  \frac{\sum_{\mu\in \mathcal{P}(n)}B(\mu)}{(q-1)^n},\quad \mbox{(by (2.5))}
\ee
as desired.
\end{proof}
\par
To apply Theorem \ref{T1} to compute the numbers of isomorphism classes of idempotent evolution algebras over a finite field, one needs to develop a method to compute the numbers $b(\mu, t)$, which we will consider next.
%%%%%%%%%%%%%%%%
\section{Computing $b(\mu, t)$}
Let $\sigma t\in G$, where $\sigma\in S_n$ and $t=(t_1, ..., t_n)\in T_n$. Recall that we also denote the corresponding permutation matrix of $\sigma$ defined by (\ref{e13}) as $\sigma$. By (\ref{e16}), $A = (a_{ij})\in X$ is fixed by $\sigma t$ if and only if
\bea\label{e31}
tA = \sigma^{-1} A \sigma t^2 \Longleftrightarrow t_ia_{ij} = a_{\sigma(i)\sigma(j)}t_j^2, \; 1\le i, j \le n.
\eea
Thus, given a partition $\mu$ of $n$ and a $t\in T_n$, the computation of $b(\mu, t)$ can be done by counting the solutions of the following system of linear equations
\bea\label{e32}
 t_ix_{ij} - t_j^2x_{\mu(i)\mu(j)} = 0, \quad1\le i, j \le n, 
\eea
that also satisfy the condition $\det (x_{ij})_{n\times n}\ne 0$. 
\par
For each $1\le i\le n$, the $n$ equations involving $x_{i1}, ..., x_{in}$ (the equations given by the $i$-th row of the matrix $(x_{ij})$) can be combined into a column as 
\bea\label{e33}
t_i(x_{i1}, ..., x_{in})^T - t^2(x_{\mu(i)\mu(1)}, ..., x_{\mu(i)\mu(n)})^T = 0,
\eea
where $t^2 = \mbox{diag}(t_1^2, ..., t_n^2)$ as before. Rewrite the above equation as 
\bea\label{e34}
t_i(x_{i1}, ..., x_{in})^T - t^2\mu^{-1}(x_{\mu(i)1}, ..., x_{\mu(i)n})^T = 0,
\eea
and let $x_i = (x_{i1}, ..., x_{in})$. Then the whole linear system (\ref{e32}) can be written as
\bea\label{e35}
\begin{pmatrix}
C_{11} &\cdots & C_{1n}\\
\vdots & {}    & \vdots\\
C_{n1} & \cdots & C_{nn}
\end{pmatrix}
\begin{pmatrix}
x_1^T \\ \vdots \\ x_n^T
\end{pmatrix} = 0,
\eea
where each $C_{ij}$ is a block matrix of size $n\times n$, and the $i$-th row of blocks is given by
\bea\label{e36}
C_{ii} &=& \begin{cases}
                   t_iI_n,&\mbox{if $\mu(i) \ne i$},\\
                   t_iI_n - t^2\mu^{-1}, & \mbox{if $\mu(i) = i$},
            \end{cases}\\
C_{ij} &=& \begin{cases}
                   0,&\mbox{if $j\ne \mu(i)$},\\
                   - t^2\mu^{-1}, & \mbox{if $j=\mu(i)$},
            \end{cases}\; (i\ne j).\nonumber
\eea
\par
Let the coefficient matrix of (\ref{e35}) be $C$. From the above formulas, we can see that the matrix $C$ is divided into blocks according to the cycle structure of $\mu$. Suppose $\mu = \{0< \mu_1\le \mu_2 \cdots\le \mu_r\}$, then $C$ is a diagonal block matrix
\bea\label{e37}
C = \begin{pmatrix}
         D_1 & {} & {}\\
         {}    &\ddots & {}\\
         {}   & {}  & D_r
      \end{pmatrix},
\eea
where each $D_i, 1\le i\le r$, is of size $n\mu_i\times n\mu_i$ determined as follows. 
\par\medskip
(1) $\mu_i = 1$. According to our notation, these $\mu_i$ come first. In this case, $D_i = C_{ii} = t_iI_n - t^2\mu^{-1}$ and all other $C_{ij}=0$ ($i\ne j$) on the same row since $\mu(i)=i$. 
\par\medskip
(2)  $\mu_i > 1$. Then according to our notation, the corresponding cycle of $\mu_i$ in $\mu$ is $(k+1, k+2, ..., k+\mu_i)$, where $k=\sum_{j<i}\mu_j$, so by (\ref{e36}),
\bea\label{e38}
D_i = \begin{pmatrix}
           t_{k+1}I_n & -t^2\mu^{-1} & {} & {}\\
           {}   & t_{k+2}I_n &  \ddots & {}\\
           {}  &   {}   &\ddots   & -t^2\mu^{-1}\\
           -t^2\mu^{-1} & {} & {} & t_{k+\mu_i}I_n
         \end{pmatrix}.
\eea
\par\medskip
Note that for each of the cases $\mu_i=1$, the $n\times n$ matrix $D_i$ has a structure that is similar to $C$. More precisely, $D_i=t_iI_n-t^2\mu^{-1} =\mbox{diag}(D_{i1}, ..., D_{ir})$ is a diagonal block matrix, such that $D_{ij} = t_i -t_j^2$ if $\mu_j =1$; and for $\mu_j>1$, 
\bea\label{e39}
\qquad D_{ij} = \begin{pmatrix}
           t_i & -t_{k+1}^2 & {} & {}\\
           {}   & t_i &  \ddots & {}\\
           {}  &   {}   &\ddots   & -t_{k+\mu_j-1}^2\\
           -t_{k+\mu_j}^2 & {} & {} & t_i
         \end{pmatrix}, \;\mbox{where}\; k=\sum_{s<j}\mu_s.
\eea
\par\medskip
\begin{example} We give two examples for the notation in the above discussions. 
\par
Let $n=3$ and let $\mu = \{0<1\le 1\le 1\}$. Then as a permutation, $\mu = (1)(2)(3) = (1)$, and as a permutation matrix, $\mu = I_3$. Thus, $C=\mbox{diag}(D_1,D_2,D_3)$, $t^2\mu^{-1} = \mbox{diag}(t_1^2, t_2^2, t_3^2)$, and 
\be
D_i = \mbox{diag}(t_i-t_1^2, t_i-t_2^2,t_i-t_3^2), \quad 1\le i\le 3.
\ee
\par\medskip
Let $n=4$ and let $\mu = \{0<1<3\}$. Then as a permutation $\mu = (1)(234)=(234)$, and as a permutation matrix (see (\ref{e13}))
\be
\quad\mu = \begin{pmatrix}
           1 & 0 & 0 & 0\\
           0 & 0 & 0 & 1 \\
           0 & 1 & 0 & 0 \\
           0 & 0 & 1 & 0
          \end{pmatrix} \;\mbox{and}\;\;
\mu^{-1} = \begin{pmatrix}
           1 & 0 & 0 & 0\\
           0 & 0 & 0 & 1 \\
           0 & 1 & 0 & 0 \\
           0 & 0 & 1 & 0
          \end{pmatrix}^T = 
          \begin{pmatrix}
           1 & 0 & 0 & 0\\
           0 & 0 & 1 & 0 \\
           0 & 0 & 0 & 1 \\
           0 & 1 & 0 & 0
          \end{pmatrix}. 
\ee
Thus $C = \mbox{diag}(D_1,D_2)$,
\be
t^2\mu^{-1} =  \begin{pmatrix}
           t_1^2 & 0 & 0 & 0\\
           0 & 0 & t_2^2 & 0 \\
           0 & 0 & 0 & t_3^2 \\
           0 & t_4^2 & 0 & 0
          \end{pmatrix},\;\; 
D_1 = \begin{pmatrix}
           t_1 - t_1^2 & 0 & 0 & 0\\
           0 & t_1 & -t_2^2 & 0 \\
           0 & 0 & t_1 & -t_3^2 \\
           0 & -t_4^2 & 0 & t_1
          \end{pmatrix},
\ee
and 
\be
D_2 = \begin{pmatrix}
           t_2I_4 & -t^2\mu^{-1} & 0 \\
           0 & t_3I_4 & -t^2\mu^{-1} \\
           -t^2\mu^{-1} & 0 & t_4I_4 
          \end{pmatrix}_{12\times 12}. 
\ee
\end{example}
%%%%%%
\par\medskip
By using elementary row operations, we can further reduce the matrices $D_i$ in (\ref{e38}) and $D_{ij}$ in (\ref{e39}). Let  
\bea\label{e310}
k_1 = 0 \;\mbox{and}\; k_i =\sum_{s<i}\mu_s, 1< i\le r. 
\eea
If $\mu_i = 1$ and $\mu_j>1$, then the matrix $D_{ij}$ in (\ref{e39}) can be row reduced to
\bea\label{e311}
D'_{ij} = \begin{pmatrix}
           t_i & -t_{k_j+1}^2 & {} & {}\\
           {}   & \ddots &  \ddots & {}\\
           {}  &   {}   &t_i   & -t_{k_j+\mu_j-1}^2\\
           0 & \cdots & 0 & m_{ij}
         \end{pmatrix}, 
\eea
where 
\bea\label{e312}
m_{ij} = t_i^{\mu_j} - \prod_{s=1}^{\mu_j}t^2_{k_j+s}.
\eea
Similarly, $D_i$ in (\ref{e38}) can be row reduced to
\bea\label{e313}
D'_i = \begin{pmatrix}
           t_{k_i+1}I_n & -t^2\mu^{-1} & {} & {}\\
           {}   &\ddots  &  \ddots & {}\\
           {}  &   {}   &t_{k_i+\mu_i-1}I_n   & -t^2\mu^{-1}\\
           0 & \cdots & 0 & R_i
         \end{pmatrix},
\eea
where 
\bea\label{e314}
R_i = (\prod_{s=1}^{\mu_i}t_{k_i+s})I_n - (t^2\mu^{-1})^{\mu_i}.
\eea
\par
We summarize our discussions in the following theorem, which further reduces the number of $t\in T_n$ we need to consider in the computation.
\begin{theorem} Given a partition $\mu = \{0<\mu_1\le \mu_2 \le\cdots\le\mu_r\}$ of $n$, assume that $\mu_i = 1$ for $i =1, ..., r'$ (if no $\mu_i=1$, then $r'=0$) and $\mu_i > 1$ for  $i =r'+1, ..., r$. In order for the linear system (\ref{e34}) to have nontrivial solutions for every $1\le i\le n$, it is necessary and sufficient that 
\bea\label{e315}
\qquad\;\;\prod_{j=1}^{r'}(t_i-t_j^2)\prod_{j=r'+1}^r(t_i^{\mu_j} - \prod_{s=1}^{\mu_j}t^2_{k_j+s}) = 0,\; \mbox{$\forall\;i$ such that $\mu_i=1$;}
\eea
and
\bea\label{e316}
\qquad\det((\prod_{j=1}^{\mu_i}t_{k_i+j})I_n - (t^2\mu^{-1})^{\mu_i})=0,\; \mbox{$\forall \;i$ such that $\mu_i>1$}.
\eea
\end{theorem}
\begin{proof} Use (\ref{e37}) to rewrite the linear system (\ref{e35}) as
\be
\begin{pmatrix}
         D_1 & {} & {}\\
         {}    &\ddots & {}\\
         {}   & {}  & D_r
      \end{pmatrix}
      \begin{pmatrix}
        x_1^T \\ \vdots \\ x_n^T
       \end{pmatrix} = 0.
\ee
In order for this system to have nontrivial solutions for {\it all} row vectors $x_i$, it is necessary that $\det(D_i) = 0, 1\le i\le r$. For those $i$ such that $\mu_i = 1$, by (\ref{e311}) and (\ref{e312})
\be
\det(D_i) &=& \det(t_iI_n-t^2\mu^{-1}) =\det(D_{i1})\cdots\det(D_{ir})\\
  {}  &=& t_i^a\prod_{j=1}^{r'}(t_i-t_j^2)\prod_{j=r'}^r(t_i^{\mu_j} - \prod_{s=1}^{\mu_j}t^2_{k_i+s}) \;\mbox{for some $a$}. 
\ee
Since $t_i\ne 0$, in these cases, $\det(D_i) = 0$ is equivalent to (\ref{e315}). Similarly, for those $i$ such that $\mu_i>1$, by (\ref{e313}) and (\ref{e314}), $\det(D_i) = 0$ is equivalent to (\ref{e316}). 
\par
These conditions are also sufficient. This is clear if $\mu_i =1$, since in this case, the linear system involves $x_i$ is $D_ix_i^T = 0$. As for $\mu_i>1$, we note that by (\ref{e313}), a nontrivial solution of $R_ix^T_{k_i+\mu_i} = 0$ will lead to nontrivial solutions for $x_{k_i+s}$ for all $s = 1, ..., \mu_i-1$.
\end{proof}
\par
\begin{example} We continue on the two examples in Example 3.1. 
\par
For the example where $n=3$ and $\mu =\{0<1\le 1\le 1\}$, the conditions in Theorem 3.1 are
\be
(t_i - t_1^2)(t_i - t_2^2)(t_i - t_3^2) = 0,\; i = 1, 2, 3.
\ee
\par
For the example where $n=4$ and  $\mu = \{0<1<3\}$, the conditions are
\be
(t_1-t_1^2)(t_1^3-t_2^2t_3^2t_4^2) &=& 0,\\
(t_2t_3t_4-t_1^6)(t_2t_3t_4 - t_2^2t_3^2t_4^2)^3 &=& 0.
\ee
Since $t_i\ne 0, 1\le i\le 4$, the above conditions are equivalent to
\be
(1-t_1)(t_1^3-t_2^2t_3^2t_4^2) &=& 0,\\
(t_2t_3t_4-t_1^6)(1 - t_2t_3t_4) &=& 0.
\ee
\end{example}
\par\medskip
For a given $\mu$, we first use Theorem 3.1 to select the set of $t$ that satisfies (\ref{e315}) and (\ref{e316}), since any other $t$ will lead to $b(\mu,t)=0$. For each of these selected $t$, we count the solutions $x_i, 1\le i\le n$, of (\ref{e35}) such that $\det(x_{ij})_{n\times n}\ne 0$. The number of these solutions is $b(\mu, t)$. Since the approach developed here involves the determination of nonzero determinants, which leads to high order multivariable polynomials equation systems in general, in the sections that follow, we restrict our attention to the special cases of dimensions $2$, $3$, and $4$. In all these cases, nice formulas for the isomorphism classes of idempotent evolution algebras can be obtained over any finite field.
%%%%%%%%%%%%%%%%%%
\section{The case $n=2$}
In this section we consider the case $n=2$. There are two partitions for $2$: $\{1,1\}$ and $\{2\}$. In both cases, $d(\mu) = 2$.
\par\medskip
\subsection{The case $\mu = \{1,1\}$} 
\par
As a matrix, $\mu = I_2$, and $t^2\mu^{-1} = t^2 = \mbox{diag}(t_1^2,t_2^2)$. The conditions imposed on $t$ by Theorem 3.1 (only (\ref{e315}) applies) are
\bea\label{e41}
(t_1 -1)(t_1-t_2^2) = 0\quad\mbox{and}\quad (t_2 -1)(t_2-t_1^2) = 0.
\eea
It follows that if one of $t_1$ and $t_2$ is $1$, then the other must also be $1$ in order for both equations in (\ref{e41}) to hold. If both $t_1, t_2 \ne 1$, then 
\be
t_1-t_2^2 = 0 = t_2-t_1^2\Rightarrow t_1 = t_1^4 
\ee
$\Rightarrow$ both $t_1, t_2$ are primitive roots of $x^3 - 1$ and $t_1^2 = t_2$. In order for $x^3 - 1$ to have $3$ distinct roots in $\mathbb{F}_q$, it is necessary and sufficient that $3|(q-1)$. Thus there are two cases for the partition $\{1,1\}$.
\par\medskip
\begin{enumerate}
\item $t = (1,1)$, $\mu t = I_2$, then $b(\mu, t) = |GL_2(\mathbb{F}_q)| = (q^2-1)(q^2-q)$.
\par\medskip
\item $t=(t_1,t_2)$, $t_1, t_2$ are primitive roots of $x^3-1$ and $t_1^2 = t_2$. This happens if and only if $3|(q-1)$. There are two choices for $t_1$, but we only need to consider one by symmetry. So consider $C =\mbox{diag}(D_1,D_2)$, where (see (\ref{e37})--(\ref{e39}))
\be
\qquad D_1 &= & \begin{pmatrix} D_{11} & 0 \\ 0 & D_{12}\end{pmatrix}  = \begin{pmatrix} t_1 - t_1^2 & 0\\ 0 & t_1 - t_2^2 \end{pmatrix} = \begin{pmatrix} t_1 - t_1^2 & 0\\ 0 & 0 \end{pmatrix},\\
\qquad D_2 &=& \begin{pmatrix} D_{21} & 0 \\ 0 & D_{22}\end{pmatrix}  = \begin{pmatrix} t_2 - t_1^2 & 0\\ 0 & t_2 - t_2^2 \end{pmatrix} = \begin{pmatrix} 0 & 0 \\ 0 & t_2^2 - t_2 \end{pmatrix}.
\ee
Thus, the solutions (see (\ref{e35}) and (\ref{e37})) are $x_1 = (0, a), a\ne 0$ and $x_2 = (b, 0), b\ne 0$. Since $\det(x_1^T,x_2^T)\ne 0$ for these solutions, $b(\mu, t) = (q-1)^2$ for each of the two $t$'s. 
\end{enumerate}
To summarize, the contribution of $\mu=\{1,1\}$ to the sum of (\ref{e26}) (defined by (\ref{e25})) is (we factor out the powers of $q-1$ since we will divide the final sum by $(q-1)^2$)
\be
B(\{1,1\}) = \begin{cases} \frac{(q-1)^2}{2}(q^2+q+2), & \mbox{if $3|(q-1)$};\\
            \frac{(q-1)^2}{2}(q^2+q), & \mbox{if $3\notdivides (q-1)$}.
            \end{cases}
\ee
\par\medskip
\subsection{The case $\mu = \{2\}$}
\par\medskip
 In this case, we have 
\bea\label{e42}
t^2\mu^{-1} = \begin{pmatrix} 0& t_1^2\\ t_2^2 & 0 \end{pmatrix},\;
(t^2\mu^{-1})^2 = \begin{pmatrix} t_1^2t_2^2 & 0 \\ 0 & t_1^2t_2^2 \end{pmatrix}.
\eea
Thus (\ref{e316}) gives
\bea\label{e43}
\det(t_1t_2I_2 - (t^2\mu^{-1})^2) = 0\;\Longleftrightarrow t_1t_2 = 1.
\eea
Thus, there are $q-1$ choices of  $t = (t_1,t_2)$ that provide nontrivial solutions to the linear system (\ref{e35}).
\par
The matrix $C$ defined by (\ref{e37}) in this case is (see (\ref{e38}))
{\small \bea\label{e44}
\qquad\begin{pmatrix}t_1I_2 & -t^2\mu^{-1}\\ -t^2\mu^{-1} & t_2I_2 \end{pmatrix},\;\mbox{row reduction $\longrightarrow$}\;
\begin{pmatrix}t_1I_2 & -t^2\mu^{-1}\\ 0 & 0 \end{pmatrix}.
\eea}
The solution $x_2^T$ in (\ref{e35}) needs to be a nonzero vector of $\mathbb{F}^2_q$, and there are a total of $q^2 -1$ of them. Write $x_2 = (x_{21}, x_{22})$, then since $t_1t_2 = 1$, 
\bea\label{e45}
x_1^T = t_1^{-1}t^2\mu^{-1}x_2^T = (t_1x_{22}, t_1^{-1}t_2^2x_{21})^T = (t_2^{-1}x_{22}, t_2^{3}x_{21})^T.
\eea
We need these solutions to satisfy
\bea\label{e46}
\det (x_1^T, x_2^T)\ne 0 \Longleftrightarrow x_{22} \ne \pm t_2^2x_{21}.
\eea
For a given $t$, we discuss the cases according to the characteristic $p$ of the field $\mathbb{F}_q$. 
\par
If $p= 2$, then we only need $x_{22} \ne t_2^2x_{21}$, and the total of the pairs $x_1, x_2$ such that $\det (x_1^T, x_2^T)\ne 0$ is $b(\mu,t) = q(q-1)$. 
\par
If $p\ne 2$, then  
\be
b(\mu,t)= (q-1)(q-2)+q-1=q(q-2)+1 =(q-1)^2,.
\ee 
This is because for $x_{21} \ne 0$, there are $q-2$ choices of $x_{22}$; and for $x_{21} = 0$, there are $q-1$. 
\par\medskip
Thus, the contribution of $\mu = \{2\}$ to the sum of (\ref{e26}) is (multiply $b(\mu, t)$ by $q-1$, the number of $t$'s)
\be
B(\{2\}) = \begin{cases} \frac{1}{2}q(q-1)^2, & \mbox{if $p=2$};\\
                                 \frac{1}{2}(q-1)^3, & \mbox{if $p\ne 2$}.
          \end{cases}
\ee
\par\medskip
Finally, the number of orbits $|X/G|$ in the case $n=2$ is then given by the formula
\be
|X/G| = \frac{1}{(q-1)^2}(B(\{1,1\})+B(\{2\})).
\ee
For example, if $p = 2$ and $3|(q-1)$, then 
\be
|X/G| &=& \frac{1}{(q-1)^2} \left(\frac{(q-1)^2}{2}(q^2+q+2)+ \frac{1}{2}q(q-1)^2\right)\\
 {} &=& \frac{1}{2}((q+1)^2 +1).
\ee
\par\medskip
We summarize the results in the following theorem.
\begin{theorem}
The number $\mathcal{N}(2,\mathbb{F}_q)$ of isomorphism classes of $2$-dimensional idempotent evolution algebras over a finite field $\mathbb{F}_q$, where $q=p^m$, is given by the table below:
\begin{table}[h]
%\caption{Number of Isomorphism Classes of $2$-dimensional Idempotent Evolution Algebras} %title of the table
\centering % centering table
\begin{tabular}{c|c|c} 
%\hline\hline %inserting double-line
%Audio Name&\multicolumn{7}{c}{Sum of Extracted Bits} \\ [0.5ex]
%\hline
{} & $3|(q-1)$ & $3\notdivides (q-1)$\\
\hline
{} & {} & {}\\
$p=2$ & $\frac{(q+1)^2}{2}+\frac{1}{2}$ & $\frac{(q+1)^2}{2}-\frac{1}{2}$\\
\hline
{} & {} & {}\\
$p\ne 2$ & $\frac{(q+1)^2}{2}$ & $\frac{(q+1)^2}{2}-1$\\
\hline
\end{tabular}
%\hline % inserts single-line
\label{t1}
\end{table}
\end{theorem}
%%%%%%%%%%%%%%%%%%
\section{The case $n=3$}
There are $3$ partitions of the integer $3$: $\{1,1,1\}, \{1,2\}$ and $\{3\}$.
\par
\subsection {The case $\mu=\{1,1,1\}$}
\par\medskip
 Given $t =(t_1,t_2,t_3)$ such that $t_1t_2t_3\ne 0$, the matrix in (\ref{e37}) is: 
 \be
 C &=& \mbox{diag}(D_1,D_2,D_3),\;\mbox{where}\\
D_i &=& \mbox{diag}(t_i-t_1^2, t_i-t_2^2, t_i - t_3^2),\;1\le i\le 3.
\ee
The conditions for all $D_ix_i^T = 0, 1\le i\le 3$, to have nontrivial solutions are (c.f. Example 3.2):
\be
(t_i - t_1^2)(t_i - t_2^2)(t_i - t_3^2) = 0,\; i = 1, 2, 3.
\ee
Note that for each $i$, one of the terms in the above equation is $t_i - t_i^2$, which cannot be $0$ unless $t_i=1$. Thus we separate the cases according to whether each of the $t_i$'s is $1$ or not. This leads to the cases listed below for further consideration, for all other choices of $t$, $b(\mu,t) = 0$ by Theorem 3.1.
\par\medskip
(5.1a). All $t_i = 1$. Note that any two $t_i=1$ implies all $t_i=1$.
\par
(5.1b). One $t_i = 1$, the other two are not equal and are both primitive roots of $x^3 - 1$. This happens if and only if $3|(q-1)$. If that is the case, there are $6$ of these $t$.
\par
(5.1c). Assume that $3|(q-1)$, and let $\xi_1, \xi_2$ be the two primitive roots of $x^3 - 1$. Any of the $6$ distinct permutations of $(\xi_1,\xi_1,\xi_2)$ and $(\xi_2,\xi_2,\xi_1)$.
\par\medskip
(5.1d). If $7|(q-1)$, then there are the following choices for $t$. All $t_i$ are primitive roots of $x^7 - 1$ such that they satisfy one the following conditions:
\bea\label{e51}
t_1^2 = t_2, t_2^2 = t_3, t_3^2 = t_1;\quad \mbox{or}\quad t_1^2 = t_3, t_2^2 = t_1, t_3^2 = t_2.
\eea
For each of the conditions in (\ref{e51}), let $t_1$ run through the primitive roots of $x^7 - 1$, we obtain $6$ choices for $t$. Thus there are a total of $12$ of these $t$.
\par\medskip
We now compute the number $b(\mu,t)$ for the cases listed above. The corresponding permutation matrix is $I_3$. %Let $e_1,e_2,e_3$ be the standard basis of $\mathbb{F}_q^3$.
\par\medskip
{\bf Case (5.1a).} All $t_i = 1$. Then $\mu t  = I_3$ and 
\be
b(\mu, t) = |X| = (q^3-1)(q^3-q)(q^3-q^2).
\ee
\par\medskip
{\bf Case (5.1b).} Assume $3|(q-1)$. By symmetry, we only need to consider the case $t_1 =1, t_2 = \xi_1, t_3 = \xi_2$. Then 
\be
D_1 = \mbox{diag}(0, 1-t_2^2,1-t_3^2) \Longrightarrow x_1 = (x_{11},0,0), x_{11}\ne 0,\\
D_2 = \mbox{diag}(t_2 - 1, t_2-t_2^2, 0) \Longrightarrow x_2 = (0,0,x_{23}), x_{23}\ne 0,\\
D_3 = \mbox{diag}(t_3 - 1, 0,t_3-t_3^2) \Longrightarrow x_3 = (0,x_{32},0), x_{32}\ne 0.
\ee 
For all these solutions, $\det(x_1^T,x_2^T,x_3^T)\ne 0$, thus $b(\mu,t) = (q-1)^3$ for any of these $6$ choices of $t$.
\par\medskip
{\bf Case (5.1c).} Assume $3|(q-1)$. By symmetry, we only need to consider the case $t_1 =\xi_1 =t_2, t_3 = \xi_2 = t_1^2$. We have
\be
D_1 &=& \mbox{diag}(t_1-t_1^2, t_1-t_1^2,0) \Longrightarrow x_1 = (0,0,x_{13}), x_{13}\ne 0,\\
D_2 &=& \mbox{diag}(t_1 - t_1^2, t_1-t_1^2, 0) \Longrightarrow x_2 = (0,0,x_{23}), x_{23}\ne 0,\\
D_3 &=& \mbox{diag}(0, 0,t_1^2-t_1) \Longrightarrow x_3 = (x_{31}, x_{32}, 0) \ne 0.
\ee 
Since $x_1$ and $x_2$ are dependent, $b(\mu,t)=0$ for these $t$.
\par\medskip
{\bf Case (5.1d).} Assume $7|(q-1)$. Similar to the discussions in (5.1b) above, we find that $b(\mu,t) = (q-1)^3$ for any of these $12$ choices of $t$.
\par\medskip
Since for $\mu = \{1,1,1\}$, $d(\mu)=3!=6$, the contribution of this $\mu$ to the sum in (\ref{e26}) is:
{\small 
\bea\label{e52}
\qquad B(\{1,1,1\}) = \begin{cases} |X|/6, &\mbox{if $3\notdivides (q-1), 7\notdivides(q-1)$};\\
                                 |X|/6+(q-1)^3, &\mbox{if $3|(q-1), 7\notdivides (q-1)$};\\
                                 |X|/6+2(q-1)^3, &\mbox{if $3\notdivides (q-1), 7|(q-1)$};\\
                                 |X|/6+3(q-1)^3, &\mbox{if $3|(q-1), 7|(q-1)$}.
           \end{cases}
\eea}
\par\medskip
\subsection{The case $\mu = \{1,2\}$}
\par\medskip
 We have
{\small
\be
\mu = \begin{pmatrix}1 & {} & {} \\ {} & 0 & 1\\ {} & 1 & 0 \end{pmatrix} &=& \mu^{-1},\quad
t^2\mu^{-1} = \begin{pmatrix}t_1^2 & {} & {} \\ {} & 0 & t_2^2\\ {} & t_3^2 & 0 \end{pmatrix},\\
(t^2\mu^{-1})^2 &=& \begin{pmatrix}t_1^4 & {} & {} \\ {} & t_2^2t_3^2 & {}\\ {} & {} & t_2^2t_3^2 \end{pmatrix}.
\ee}
Thus the two equations given by (\ref{e315}) and (\ref{e316}) are
\be
{} & {} & (t_1-t_1^2)(t_1^2 - t_2^2t_3^2) = 0 \;\mbox{and}\; \det(t_2t_3I_3 - (t^2\mu^{-1})^2) = 0\\
{} &\Leftrightarrow &   (1-t_1)(t_1^2 - t_2^2t_3^2) = 0 \;\mbox{and}\; (t_2t_3-t_1^4)(1-t_2t_3) = 0.
\ee
These equations lead to the cases listed below for further consideration, for all other $t$, $b(\mu, t)=0$:
\par\medskip
(5.2a). $t_1=1$ and $t_2t_3 = 1$. There are $q-1$ of these $t$.
\par
(5.2b). $p\ne 2$, $t_1=-1$ and $t_2t_3 = 1$.
\par
(5.2c). $3|(q-1)$, $t_1$ is a primitive root of $x^3-1$ and $t_2t_3 = t_1$.
\par
(5.2d). $p\ne 2$, $t_1\ne -1$, $t_1^3 = -1$, and $t_2t_3 = -t_1$.
\par\medskip
We will see that for (5.2b), (5.2c) and (5.2d), $b(\mu,t)=0$, so there is no need to count these $t$. The matrix $C$ of (\ref{e37}) is $C =\mbox{diag}(D_1, D_2)$, where
{\small
\bea\label{e53}
D_1 &=&  t_1I_3 - t^2\mu^{-1} = \begin{pmatrix}t_1-t_1^2 & {} & {}\\ {} & t_1 & -t_2^2 \\ {} & -t_3^2 & t_1\end{pmatrix},\\ \nonumber
& &\\ \nonumber
D_2 &=& \begin{pmatrix}t_2I_3 & -t^2\mu^{-1}\\ -t^2\mu^{-1} & t_3I_3\end{pmatrix}.\nonumber
\eea}
They can be row reduced to
{\small 
\bea\label{e54}
D_1 &\rightarrow& D_1'=\begin{pmatrix}1-t_1 & {} & {}\\ {} & t_1 & -t_2^2 \\ {} & 0 & t_1^2-t_2^2t_3^2\end{pmatrix},\\\nonumber
&&\\\nonumber
D_2 &\rightarrow& D_2'=\begin{pmatrix}t_2I_3 & -t^2\mu^{-1}\\ 0 & t_2t_3I_3-(t^2\mu^{-1})^2\end{pmatrix}.\nonumber
\eea}
Since $t_1t_2t_3\ne 0$, we can perform further reduction 
\par\medskip
{\small 
\bea\label{e55}
\quad t_2t_3I_3-(t^2\mu^{-1})^2 \rightarrow \mbox{diag}(t_2t_3 - t_1^4, 1-t_2t_3, 1-t_2t_3) =: D_3.
\eea}
\par\medskip
{\bf Case (5.2a).} For $t_1=t_2t_3 = 1$, $D_3 = 0$. Thus $x_3 = (x_{31}, x_{32}, x_{33})\ne 0$ is arbitrary, and (use $D_2'$) 
\be
x_2^T = t_2^{-1}t^2\mu^{-1}x_3^T = t_2^{-1}(t_1^2x_{31}, t_2^2x_{33}, t_3^2x_{32})^T.
\ee
Up to a nonzero scalar multiple, we can just assume $x_2 = (x_{31}, t_2^2x_{33}, t_2^{-2}x_{32})$. We also have
{\small
\be
D_1' = \begin{pmatrix}0 & {} & {}\\ {} & 1 & -t_2^2 \\ {} & {} & 0\end{pmatrix},
\ee}
so $x_1 = (x_{11}, t_2^2x_{13}, x_{13})$ with $x_{11}$ and $x_{13}$ not both $0$. We need the following matrices to be nonsingular (multiply column $3$ by $t_2^2$ then add the negative of row $2$ to row $3$ in the reduction steps)
{\footnotesize
\be
\begin{pmatrix}x_{11} & t_2^2x_{13} & x_{13}\\ x_{31} & t_2^2x_{33} &t_2^{-2}x_{32}\\
                              x_{31} & x_{32} & x_{33}\end{pmatrix} \rightarrow 
           \begin{pmatrix}x_{11} & t_2^2x_{13} & t_2^2x_{13}\\ x_{31} & t_2^2x_{33} &x_{32}\\
                              x_{31} & x_{32} & t_2^2x_{33}\end{pmatrix} \rightarrow 
      \begin{pmatrix}x_{11} & t_2^2x_{13} & t_2^2x_{13}\\ x_{31} & t_2^2x_{33} &x_{32}\\
                              0 & x_{32} -t_2^{2}x_{33} & t_2^2x_{33}-x_{32}\end{pmatrix}.
\ee}
So $x_{32} -t_2^{2}x_{33} \ne 0$. Under this assumption, the matrix can be further reduced to (add column $3$ to column $2$)
{\small
\bea\label{e56}
 \begin{pmatrix}x_{11} & t_2^2x_{13} & t_2^2x_{13}\\ x_{31} & t_2^2x_{33} &x_{32}\\
                              0 & 1 & -1\end{pmatrix}\rightarrow
    \begin{pmatrix}x_{11} & 2t_2^2x_{13} & t_2^2x_{13}\\ x_{31} & t_2^2x_{33} +x_{32} &x_{32}\\
                              0 & 0 & -1\end{pmatrix}.
\eea}
Thus the conditions on the $5$ parameters $x_{11}, x_{13}, x_{31}, x_{32}, x_{33}$ are: 
\be
x_{32}- t_2^2x_{33}\ne 0\; \mbox{and}\; x_{11}(t_2^2x_{33} +x_{32})- 2t_2^2x_{13}x_{31}\ne 0. 
\ee
The second condition is equivalent to 
\be
u :=(x_{31}, t_2^2x_{33} +x_{32})\ne 0 \;\mbox{ and $(x_{11},2t_2^2x_{13})$ is not a multiple of $u$}. 
\ee
\par
If $p=2$, then $2t_2^2x_{13}x_{31}=2t_2^2x_{13}=0$, and the conditions reduce to $x_{11}\ne 0$ and $x_{32}\ne t_2^2x_{33}$. The total number of $x_{11}, x_{13}, x_{31}, x_{32}, x_{33}$ that satisfy these two conditions is $q^3(q-1)^2$.
\par
Assume $p\ne 2$. The number of $x_3$ such that $x_{32}\ne t_2^2x_{33}$ is $q^2(q-1)$. Among these $x_3$'s, there are $q-1$ make $u=0$. This is because if $u=0$, then $x_{31}=0$, and so $x_{33}\ne 0$; otherwise, $t_2^2x_{33}+x_{32}=0 \Rightarrow x_{32} = 0$, contradicts $x_{32}\ne t_2^2x_{33}$. Thus the number of $x_3$ such that $x_{32}\ne t_2^2x_{33}$ and $u\ne 0$ is 
\be
q^2(q-1)-(q-1) = (q-1)^2(q+1).
\ee
 The number of $(x_{11}, 2t_2^2x_{13})$ that are not multiples of a given $u$ is $q^2-q=q(q-1)$. Thus, the total number of $x_{11}, x_{13}, x_{31}, x_{32}, x_{33}$ that satisfy the conditions is 
 \be
(q-1)^2(q+1)\cdot q(q-1)=q(q-1)^3(q+1).
\ee
\par
Thus for each given $t$ such that $t_1=t_2t_3 = 1$, 
\bea\label{ea1}
b(\mu, t) = \begin{cases} q^3(q-1)^2, & \mbox{if $p=2$};\\
                q(q-1)^3(q+1),& \mbox{if $p\ne 2$}.
                \end{cases}
\eea
\par\medskip
{\bf Case (5.2b).} $p\ne 2$, $t_1=-1$ and $t_2t_3 = 1$. In this case, the matrices $D_1'$ and $D_2'$ in (\ref{e54}) are
\be
D_1'=\begin{pmatrix}2 & {} & {}\\ {} & -1 & -t_2^2 \\ {} & 0 & 0\end{pmatrix}_{3\times 3},\;
D_2'=\begin{pmatrix}t_2I_3 & -t^2\mu^{-1}\\ 0 & 0\end{pmatrix}_{6\times 6}.
\ee 
Thus $x_1=(0, -t_2^2x_{13}, x_{13})\ne 0$, $x_3=(x_{31},x_{32},x_{33})\ne 0$ arbitrary, and 
\be 
x_2^T=t_2^{-1} t^2\mu^{-1}x_3^T&=& (t_2^{-1}t_1^2x_{31},t_2x_{33}, t_2^{-1}t_3^2x_{32})^T\\
&\rightarrow& 
(x_{31},t_2^2x_{33}, t_2^{-2}x_{32})^T. 
\ee
So there are $4$ parameters $x_{13}, x_{31}, x_{32}, x_{33}$ and the matrix that needs to be nonsingular is (in the reduction, first add a multiple of column $3$ to column $2$, then add a multiple of row $1$ to row $2$)
{\small
\be
\begin{pmatrix} x_{31} & x_{32} & x_{33}\\ x_{31} & t_2^2x_{33} & t_2^{-2}x_{32}\\
                        0 &  -t_2^2x_{13} & x_{13} \end{pmatrix} &\rightarrow& 
 \begin{pmatrix} x_{31} & x_{32}+t_2^2x_{33} & x_{33}\\ x_{31} & t_2^2x_{33}+x_{32} & t_2^{-2}x_{32}\\
                        0 &  0 & x_{13} \end{pmatrix}\\
                        &\rightarrow&
\begin{pmatrix} x_{31} & x_{32} & x_{33}\\ 0 & 0 & t_2^{-2}x_{32}-x_{33}\\
                        0 &  0 & x_{13} \end{pmatrix},
\ee}
which cannot have rank $3$. Thus there is no fixed point for these $t$. 
\par\medskip
{\bf Cases (5.2c) and (5.2d).} Suppose that $t_1$ is a primitive root of $x^3 - 1$ and $t_2t_3 = t_1$. Then
\be
D_1'=\begin{pmatrix}1-t_1 & {} & {}\\ {} & -1 & -t_2^2 \\ {} & 0 & 0\end{pmatrix},\;
D_2'=\begin{pmatrix}t_2I_3 & -t^2\mu^{-1}\\ 0 & D_3\end{pmatrix},
\ee
where (see (\ref{e55})):
\be 
D_3 = \mbox{diag}(t_2t_3 - t_1^4, 1-t_2t_3, 1-t_2t_3) =  \mbox{diag}(0, 1-t_2t_3, 1-t_2t_3).
\ee 
Thus, we have $x_1 = (0, t_2^2x_{13}, x_{13}), x_3 = (x_{31}, 0, 0)$, and $x_2= (t_2^{-1}t_1^2x_{31}, 0, 0)$. Since $x_2$ and $x_3$ are dependent, $b(\mu, t) = 0$ for these $t$. Similarly $b(\mu, t) = 0$ in the case (5.2d). 
\par\medskip
For $\mu = \{1,2\}$, $d(\mu)=2$, so its contribution $B(\mu)$ to the sum in (\ref{e26}) is (multiply (\ref{ea1}) by $q-1$, the number of $t$): 
\bea\label{e57}
B(\{1,2\})=\begin{cases} q^3(q-1)^3/2,&\mbox{ if $p=2$},\\ 
                               q(q-1)^4(q+1)/2, &\mbox{ if $p\ne 2$}.
\end{cases}
\eea
\par\medskip
\subsection{The case $\mu = \{3\}$} We have
\par\medskip
{\footnotesize
\be
\mu^{-1} = \begin{pmatrix}{} & 1 & {} \\ {} & {} & 1\\ 1 & {} & {} \end{pmatrix}, \;
t^2\mu^{-1} = \begin{pmatrix}{} & t_1^2 & {}  \\ {} & {} & t_2^2 \\ t_3^2 & {} & {} \end{pmatrix},\;
(t^2\mu^{-1})^3 = (t_1t_2t_3)^2I_3.
\ee}
\par\medskip
Condition (\ref{e316}) in Theorem 3.1 is 
\be
\det(t_1t_2t_3I_3 -(t_1t_2t_3)^2I_3) = 0\Leftrightarrow t_1t_2t_3 = 1.
\ee 
Under this condition, the matrix $C = (D_1)$ in (\ref{e37}) reduces to (see (\ref{e313}))
\be
D_1' = \begin{pmatrix}t_1I_3 & -t^2\mu^{-1} & {} \\ {} & t_2I_3 & -t^2\mu^{-1} \\ {} & {} & 0\end{pmatrix}.
\ee
Thus $x'_3 = (x_{31}, x_{32}, x_{33})\ne 0$ is arbitrary and 
\be
x'_2 &=& t_2^{-1}(t_1^2x_{32}, t_2^2x_{33}, t_3^2x_{31}),\\ 
x'_1 &=& t_1^{-1}t_2^{-1}(t_1^2t_2^2x_{33}, t_2^2t_3^2x_{31}, t_1^2t_3^2x_{32}).
\ee
We need the following matrix 
{\small
\be
((x'_3)^T,(t_2x'_2)^T,(t_1t_2x'_1)^T) = \begin{pmatrix}x_{31} & t_1^2x_{32} & t_1^2t_2^2x_{33}\\
                                                             x_{32} & t_2^2x_{33} & t_2^2t_3^2x_{31}\\
                                                             x_{33} & t_3^2x_{31} & t_1^2t_2^3x_{32} \end{pmatrix}
\ee}
to be nonsingular. Since $t_1t_2t_3 = 1$, multiply the second row by $t_1^2$ and multiply the third row by $t_1^2t_2^2$, we obtain the following circulant matrix \cite[Def.~1]{Kra}:
{\small
\bea\label{e58}
C'=(x_3^T,x_2^T,x_1^T) := \begin{pmatrix}x_{31} & t_1^2x_{32} & t_1^2t_2^2x_{33}\\
                                                             t_1^2x_{32} & t_1^2t_2^2x_{33} & x_{31}\\
                                                             t_1^2t_2^2x_{33} & x_{31} & t_1^2x_{32} \end{pmatrix}.
\eea}
The polynomial (in the indeterminate $y$) representer for C’ is \cite[Def.~3]{Kra} 
\be
f(y)=x_{31} + t_1^2x_{32}y+ t_1^2t_2^2x_{33}y^2.
\ee 
It is known that $C'$ is nonsingular if and only if $\mbox{gcd}(f(y), y^3-1) =1$ \cite[Cor.~10]{Kra} (\cite{Kra} deals with complex numbers only, but the same conclusion holds for an arbitrary finite field \cite[Cor.~1]{Day} ). Since $y^3-1 = (y-1)(y^2+y+1)$, $\mbox{gcd}(f(y), y^3-1) =1$ if and only if 
{\small
\bea\label{e59}
\qquad f(1)=x_{31} + t_1^2x_{32}+ t_1^2t_2^2x_{33} \ne 0\;\mbox{and}\; 
\mbox{gcd}(f(y), y^2+y+1) = 1. 
\eea}
We consider whether or not $y^2+y+1$ is irreducible in $\mathbb{F}_q$.
\par\medskip
Assume that $y^2+y+1$ is reducible. If $p = 3$, $y^3-1 = (y-1)^3$, then we only need $f(1)\ne 0$. In this case, there are  $q^2(q-1)$ of $x_3$ such that $C'$ is nonsingular. If $p\ne 3$, $y^3-1$ has $3$ distinct roots in $\mathbb{F}_q$ by assumption, which implies $3|(q-1)$. Let $\xi$ be a primitive root of $y^3-1$, then we need $f(1)\ne 0$, $f(\xi)\ne 0$, and $f(\xi^2)\ne 0$. These conditions lead to the fact that
{\small
\be
\begin{pmatrix} 1 & 1 & 1\\ 1 & \xi & \xi^2\\ 1 & \xi^2 & \xi \end{pmatrix}
\begin{pmatrix} x_{31} \\ t_1^2x_{32} \\ t_1^2t_2^2x_{33} \end{pmatrix} \in (\mathbb{F}_q^{\times})^3,\;\mbox{where $\mathbb{F}^{\times}=\mathbb{F}-\{0\}$}.
\ee}
Thus the set of $x_3$ such that $C'$ is nonsingular has the same number of elements as $(\mathbb{F}_q^{\times})^3$, and there are $(q-1)^3$ of these $x_3$.
\par
Assume that $y^2+y+1$ is irreducible. Then $\mbox{gcd}(f(y), y^2+y+1) = 1$ $\Leftrightarrow$ $f(y)$ is not a multiple of $y^2+y+1\Leftrightarrow$  
\bea\label{e510}
x_{31}=t_1^2x_{32}=t_1^2t_2^2x_{33}
\eea 
do not hold simultaneously. The number of $x_3$ that satisfy the linear system (\ref{e510}) is $q$. Note that if (\ref{e510}) holds, then $f(1) = 0$ if and only if $x_3 = 0$ (since $p\ne 3$). There are $q^2$ of $x_3$ that satisfy $f(1) =0$ including $x_3=0$. Thus when $y^2+y+1$ is irreducible, the number of $x_3$ such that $C'$ is nonsingular is:
\be
q^3 - q - q^2 + 1 = (q-1)^2(q+1).
\ee
\par\medskip
To summarize, for $\mu = \{3\}$, in order to have fixed points, $t_1t_2t_3 = 1$, and there are $(q-1)^2$ of these $t$. If $p=3$, then $b(\mu,t) = q^2(q-1)$; if $3|(q-1)$, then $b(\mu,t) = (q-1)^3$; and if $p\ne 3$ and $3\notdivides (q-1)$, then $b(\mu,t) = (q-1)^2(q+1)$. For this $\mu$, $d(\mu)=3$, so its contribution $B(\mu)$ to the sum in (\ref{e26}) is:  
\bea\label{e511}
\qquad B(\{3\}) = \begin{cases} q^2(q-1)^3/3, &\mbox{if $p=3$};\\
                                  (q-1)^5/3, &\mbox{if $p\ne 3$ and $3|(q-1)$};\\
                                  (q-1)^4(q+1)/3, & \mbox{if $p\ne 3$ and $3\notdivides (q-1)$}.
          \end{cases}
\eea
\par\medskip
Finally, the number of isomorphism classes of $3$-dimensional idempotent evolution algebra over $\mathbb{F}_q$ is computed by 
\be
(B(\{1,1,1\})+B(\{1,2\})+B(\{3\}))/(q-1)^3.
\ee 
For example, if $p=2$, $3|(q-1)$ and $7|(q-1)$, then the number of isomorphism classes is 
{\small
\be
\mathcal{N}(3,\mathbb{F}_q) &=& \frac{1}{(q-1)^3}\left(\frac{|X|}{6} +3(q-1)^3 +\frac{q^3(q-1)^3}{2} +\frac{(q-1)^5}{3} \right)\\
&=& \frac{q^3(q+1)(q^2+q+1)}{6} +3 +\frac{q^3}{2} +\frac{(q-1)^2}{3}.
\ee}
\par\medskip
We summarize the discussions in the following theorem.
\par\medskip 
%\newpage
\begin{theorem}\label{T5} The number of isomorphism classes $\mathcal{N}(3,\mathbb{F}_q)$ of $3$-dimensional idempotent evolution algebras over a finite field $\mathbb{F}_q$, where $q = p^m$, is given by the tables below ($c=q^3(q+1)(q^2+q+1)/6$ in the tables):
\par
{\small
\begin{table}[h]
\centering % centering table
\begin{tabular}{c|c|c} 
%\hline\hline %inserting double-line
%Audio Name&\multicolumn{7}{c}{Sum of Extracted Bits} \\ [0.5ex]
%\hline
$p=2$ & $3|(q-1)$ & $3\notdivides (q-1)$\\
\hline
{} & {} & {}\\
$7|(q-1)$ & $c+3+\frac{q^3}{2}+\frac{(q-1)^2}{3}$ & $c+2+\frac{q^3}{2}+\frac{q^2-1}{3}$\\
\hline
{} & {} & {}\\
$7\notdivides (q-1)$ & $c+1+\frac{q^3}{2}+\frac{(q-1)^2}{3}$ & $c+\frac{q^3}{2}+\frac{q^2-1}{3}$\\
\hline
\end{tabular}
%\hline % inserts single-line
\label{t2}
\end{table}}
\par
{\small
\begin{table}[h]
\centering % centering table
\begin{tabular}{c|c} 
%\hline\hline %inserting double-line
%Audio Name&\multicolumn{7}{c}{Sum of Extracted Bits} \\ [0.5ex]
$p=3$ & $(\Rightarrow 3\notdivides (q-1))$\\
\hline
{} & {}\\
$7|(q-1)$ & $c+2+\frac{q(q^2-1)}{2}+\frac{q^2}{3}$\\
\hline
{} & {} \\
$7\notdivides (q-1)$ & $c+\frac{q(q^2-1)}{2}+\frac{q^2}{3}$ \\
\hline
\end{tabular}
%\hline % inserts single-line
\label{t3}
\end{table}}
{\small
\begin{table}[h]
\centering % centering table
\begin{tabular}{c|c|c} 
%\hline\hline %inserting double-line
%Audio Name&\multicolumn{7}{c}{Sum of Extracted Bits} \\ [0.5ex]
%\hline
$p>3$ & $3|(q-1)$ & $3\notdivides (q-1)$\\
\hline
{} & {} & {}\\
$7|(q-1)$ & $c+3+\frac{q(q^2-1)}{2}+\frac{(q-1)^2}{3}$ & $c+2+\frac{q(q^2-1)}{2}+\frac{q^2-1}{3}$ \\
\hline
{} & {} & {}\\
$7\notdivides (q-1)$ & $c+1+\frac{q(q^2-1)}{2}+\frac{(q-1)^2}{3}$  & $c+\frac{q(q^2-1)}{2}+\frac{q^2-1}{3}$ \\
\hline
\end{tabular}
%\hline % inserts single-line
\label{t4}
\end{table}}
\end{theorem} 
\par\medskip
%%%%%%%%%%%%%%%%%%%%%
\section{The case $n=4$}
There are $5$ partitions of the integer $4$: $\{1,1,1,1\}, \{1,1,2\}$, $\{1,3\}$, $\{2,2\}$, and $\{4\}$. The corresponding $d(\mu)$ are $4!, 4, 3, 8, 4$, respectively. 
\par
\subsection {The case $\mu=\{1,1,1,1\}$}
\par\medskip
For a given $t =(t_1,t_2,t_3, t_4), t_1t_2t_3t_4\ne 0$,the matrix in (\ref{e37}) is $C = \mbox{diag}(D_1,D_2,D_3,D_4)$, where 
\bea\label{e611}
D_i =\mbox{diag}(t_i-t_1^2, t_i-t_2^2, t_i - t_3^2, t_i-t_4^2),\;1\le i\le 4.
\eea
The conditions for all $D_ix_i^T = 0, 1\le i\le 4$, to have nontrivial solutions are
\be
(t_i - t_1^2)(t_i - t_2^2)(t_i - t_3^2)(t_i-t_4^2) = 0,\; 1\le i \le 4.
\ee
As in the case of $n=3$, we discuss the cases according to whether $t_i =1$ or not. This leads to the cases listed below for further consideration. For all other choices of $t$, $b(\mu,t) = 0$.
\par\medskip
{\bf (6.1a).} All $t_i = 1$. Note that if $3$ of the $t_i$ are $1$, then the fourth must also be $1$. In this case, $\mu t= I_4$ and 
\bea\label{e612}
b(\mu, t) &=& |X| = (q^4-1)(q^4-q)(q^4-q^2)(q^4-q^3)\\
                   {} &=& q^6(q-1)^4(q+1)^2(q^2+1)(q^2+q+1). \nonumber
\eea
\par\medskip
{\bf (6.1b).} Two $t_i = 1$, the other two are not equal (otherwise they are also equal to $1$) and are both primitive roots of $x^3 - 1$. This happens if and only if $3|(q-1)$. In this case, there are $12$ of these $t$. By symmetry, we just need to consider the case $t=(1,1,t_3,t_3^2)$, where $t_3$ is a primitive root of $x^3-1$. The matrices $D_i$ in (\ref{e611}) and their corresponding solutions are ($\ast$ indicates something nonzero): 
\be
D_1 = \mbox{diag}(0,0,\ast,\ast)\;&\longrightarrow& x_1 = (x_{11}, x_{12}, 0,0),\\
D_2 = \mbox{diag}(0,0,\ast,\ast)\;&\longrightarrow& x_2 = (x_{21}, x_{22}, 0,0),\\
D_3 = \mbox{diag}(\ast,\ast,\ast,0)\;&\longrightarrow& x_3 = (0,0,0,x_{34}),\\
D_4 = \mbox{diag}(\ast,\ast, 0,\ast)\;&\longrightarrow& x_4 = (0,0,x_{43},0).
\ee
Counting the number of these $x_1, x_2, x_3, x_4$ that are linearly independent, we obtain
\bea\label{e613}
b(\mu, t) = (q^2-1)(q^2-q)(q-1)^2=q(q-1)^4(q+1)
\eea
for each of these $t$.
\par\medskip
{\bf (6.1c).} One $t_i=1$, the other three $t_i$ are not equal to $1$. There are two possible subcases for these $t$.
\par
(6.1c1). If $7|(q-1)$, then there is the following possibility: the three $t_i$ that are not $1$ are different primitive roots of $x^7 -1$ and satisfy a cyclic relation (otherwise, $t$ will be in the case (6.1c2) below) such as $t_1=t_3^2, t_2 = t_1^2, t_3 = t_2^2$. Since there are $4$ possible choices for a $t_i$ to be  $1$, and for a fixed $t_i=1$, there are $2$ ways to choose a directed cycle of length $3$ for the other $t_i$, the number of these $t$ is $48$ ($6$ primitive roots). By symmetry, we consider the case $t =(1, t_2,t_2^4, t_2^2)$, where $t_2$ is a primitive root of $x^7-1$. For this $t$, we have 
{\small \be
\quad x_1=(x_{11},0,0,0), x_2 = (0,0,x_{21},0), x_3=(0,0,0,x_{34}), x_4=(0,x_{42},0,0).
\ee}
Thus for each of these $t$, we have 
\bea\label{e614}
b(\mu, t) = (q-1)^4.
\eea
\par
(6.1c2). If $3|(q-1)$, then there is the following possibility: the three $t_i$ that are not $1$ are primitive roots of $x^3 -1$ such that two of them are equal, and there are $24$ of these $t$. However, these $t$ lead to $b(\mu,t)=0$, since the two equal $t_i$ lead to dependent solutions for the corresponding $x_i$.
\par\medskip
{\bf (6.1d).} No $t_i = 1$. In this case, there are several possibilities for these $t$. Note that there are at least two of the $t_i$ must satisfy a circular relation, i.e. $t_i = t_j^2, t_j = t_i^2$.
\par\medskip
(6.1d1). All $t_i$ satisfy a circular relation such as 
\be
t_1 = t_2^2, t_2 = t_3^2, t_3 = t_4^2, t_4 = t_1^2.
\ee
These $t$ can be further divided into the following subcases.
\par
(6.1d1.1). Assume $3|(q-1)$. All $t_i$ are primitive roots of $x^3-1$ and the number of these $t$ is $6$ (this is equal to the number of all configurations of $(1,2,1,2)$). By using (\ref{e611}), we can see that for each of these $t$, the number of fixed points is equal to number of elements in $GL_2(\mathbb{F}_q)\times GL_2(\mathbb{F}_q)$. For example, for the configuration $(1,2,1,2)$, the matrices $D_i$ in (\ref{e611}) and their corresponding solutions are ($\ast$ indicates something nonzero): 
\be
D_1 = \mbox{diag}(\ast,0,\ast,0)\;&\longrightarrow& x_1 = (0, x_{12}, 0,x_{14}),\\
D_2 = \mbox{diag}(0,\ast, 0,\ast)\;&\longrightarrow& x_2 = (x_{21}, 0,x_{23},0),\\
D_3 = \mbox{diag}(\ast,0,\ast,0)\;&\longrightarrow& x_3 = (0,x_{32},0,x_{34}),\\
D_4 = \mbox{diag}(0, \ast, 0,\ast)\;&\longrightarrow& x_4 = (x_{41},0,x_{43},0).
\ee
Thus $x_i, 1\le i\le 4$, are independent if and only if both 
\be
\begin{pmatrix} x_{12} & x_{14}\\ x_{32} & x_{34}\end{pmatrix}\quad \mbox{and}\quad
\begin{pmatrix} x_{21} & x_{23}\\ x_{41} & x_{43}\end{pmatrix}
\ee
are nonsingular. Therefore in this case,
\bea\label{e615}
b(\mu, t) = (q-1)^4q^2(q+1)^2.
\eea
\par
(6.1d1.2). Assume $5|(q-1)$. All $t_i$ are primitive roots of $x^5-1$ and the number of these $t$ is $24$ (all permutations of the $4$ primitive roots). For each of these $t$, we have 
\bea\label{e616} 
b(\mu, t) = (q-1)^4.
\eea
\par
(6.1d1.3). Assume $15|(q-1)$. All $t_i$ are primitive roots of $x^{15}-1$ and the number of these $t$ is $48$. We also have 
\bea\label{e617}
b(\mu, t) = (q-1)^4.
\eea
\par\medskip
(6.1d2). Assume $7|(q-1)$. Three of the $t_i$ satisfy a circular relation, all $t_i$ are primitive roots of $x^7-1$, and there are $78$ of these $t$. All these $t$ lead to $b(\mu, t) = 0$. 
\par\medskip
(6.1d3). Assume $3|(q-1)$. The maximum number of the $t_i$ satisfy a circular relation is $2$, the $t_i$ are primitive roots of $x^3-1$, and there are $14$ of these $t$. Among them, $6$ with two circular relations have been considered in (6.1d1.1), and the other $8$ remaining $t$ with one circular relation of length $2$ lead to $b(\mu, t) = 0$.
\par\medskip
To summarize our discussions, we introduce the following factor indication function to simplify our notation. Let $m>1$ be an integer. We set
\bea\label{e618}
P_m=P_m(q-1) :=\begin{cases} 1, &\mbox{if $m|(q-1)$};\\
                                                         0,  &\mbox{otherwise}. \end{cases}
\eea
Multiply the $b(\mu, t)$'s given by (\ref{e612})-(\ref{e617}) by their corresponding numbers of $t$, then sum them up and divide by $4!$, we have the contribution of $\mu = \{1,1,1,1\}$ to the sum in (\ref{e26}):
{\small\bea\label{e619}
\quad B(\{1,1,1,1\}) &:=& \frac{1}{4!}[q^6(q-1)^4(q+1)^2(q^2+1)(q^2+q+1)\\ \nonumber
          &  & +P_3(12q(q-1)^4(q+1)+6q^2(q-1)^4(q+1)^2)\\ \nonumber
          &  &  +48(P_7(q-1)^4 +P_{15}(q-1)^4) + 24P_5(q-1)^4]. \nonumber
\eea}
For example, if none of $3,5,7$ divides $q-1$, then 
\be
B(\{1,1,1,1\}) = \frac{1}{4!}q^6(q-1)^4(q+1)^2(q^2+1)(q^2+q+1).
\ee
%%%%%%%%% 
\par\medskip
\subsection {The case $\mu=\{1,1,2\}$} 
The matrix in (\ref{e37}) is $C = \mbox{diag}(D_1,D_2,D_3)$, where 
{\footnotesize\be
D_i =\begin{pmatrix}t_i-t_1^2 & {} & {}\\ {} & t_i-t_2^2 & {} & {}\\{} & {} & t_i & -t_3^2\\ {} & {} & -t_4^2 & t_i\end{pmatrix},\;i=1,2;\;
D_3 = \begin{pmatrix} t_3I_4& -t^2\mu^{-1}\\ -t^2\mu^{-1} & t_4I_4 \end{pmatrix}.
\ee}
The conditions in Theorem 3.1 lead to the following system of equations:
\bea\label{e620}
(t_i - t_1^2)(t_i - t_2^2)(t_i^2-t_3^2t_4^2) & = & 0,\; i = 1, 2;\\ \nonumber
(t_3t_4-t_1^4)(t_3t_4-t_2^4)(1-t_3t_4) &= & 0. \nonumber
\eea
These equations in turn lead to the cases listed below for further consideration. For all other choices of $t$, $b(\mu,t) = 0$.
\par\medskip
{\bf (6.2a).} $t_1=t_2=t_3t_4 = 1$. There are $q-1$ of these $t$. Given one such $t$, we row reduce $D_1, D_2, D_3$, respectively, to
{\small\be
D_1' =D_2'=\begin{pmatrix}0 & {} & {} & {}\\ {} & 0 & {} & {}\\ 
                                {} & {} & 1 & -t_3^2\\ {} & {} & 0 & 0\end{pmatrix},\;\;
D_3'=\begin{pmatrix} t_3I_4& -t^2\mu^{-1}\\ 0 & 0 \end{pmatrix}. 
\ee}
From these reduced matrices, we obtain the solutions of (\ref{e35}), then we form a matrix with the rows given by these solutions and perform row reduction: first add $-1$ times row $3$ to row $4$, then add $t_3^2$ times column $4$ to column $3$ (note that $t_4 = t_3^{-1}$):
{\footnotesize\be
\begin{pmatrix} x_{11} & x_{12} & t_3^2x_{14} & x_{14} \\ x_{21} & x_{22} & t_3^2x_{24} & x_{24} \\
x_{41} & x_{42} & t_3^2x_{44} & t_3^{-2} x_{43}\\ x_{41} & x_{42} & x_{43} & x_{44} \end{pmatrix}
\longrightarrow 
\begin{pmatrix} x_{11} & x_{12} & 2t_3^2x_{14} & x_{14} \\ x_{21} & x_{22} & 2t_3^2x_{24} & x_{24} \\
x_{41} & x_{42} & t_3^2x_{44}+x_{43} & t_3^{-2} x_{43}\\ 0 & 0 & 0 & x_{44}-t_3^{-2}x_{43} \end{pmatrix}.
\ee}
Thus we need $x_{44}-t_3^{-2}x_{43}\ne 0$ and the following matrix to have rank $3$: 
{\small\be
A=\begin{pmatrix} x_{11} & x_{12} & 2t_3^2x_{14} \\ x_{21} & x_{22} & 2t_3^2x_{24} \\
x_{41} & x_{42} & t_3^2x_{44}+x_{43}  \end{pmatrix}.
\ee}
Next, we divide the discussion into two cases according to $p=2$ or not.
\par
If $p = 2$, then $2t_3^2x_{14} = 2t_3^2x_{24}=0$, so we need 
\be
x_{44}-t_3^{-2}x_{43} = x_{44}+t_3^{-2}x_{43}\ne 0, 
\ee
and $(x_{11}, x_{12}), (x_{21},x_{22})$ are linearly independent. There are 
\be
(q^2-1)(q^2-q)
\ee
linearly independent pairs of $(x_{11}, x_{12}), (x_{21},x_{22})$. The number of $(x_{43}, x_{44})$ such that $x_{44}-t_3^{-2}x_{43}\ne 0$ is $q(q-1)$. The variables $x_{14}, x_{24}, x_{41}, x_{42}$ are free. Thus if $p=2$, 
\bea\label{e621}
b(\mu,t) = q^6(q-1)^3(q+1).
\eea
\par
If $p\ne 2$, then since $2t_3^2$ is invertible, the number of $A$ that have rank $3$ can be obtained similar to the counting of the elements in $GL_3(\mathbb{F}_q)$. That is, the first row of $A$ can be any nonzero vector, and there are $q^3-1$ of them. The second row must not be a multiple of the first row, and there are $q^3-q$ of them. Now the third row is slightly different from the usual case, we need to count the number of $x_4=(x_{41}, x_{42}, x_{43}, x_{44})$ such that $x_{44}-t_3^{-2}x_{43}\ne 0$ (total $q^3(q-1)$) and $(x_{41}, x_{42}, t_3^2x_{44}+x_{43})$ is not a linear combination of the first two rows. Given a linear combination of the first two rows, the number of $x_4$ such that $(x_{41}, x_{42}, t_3^2x_{44}+x_{43})$ is equal to this vector is $q$. So we need to subtract $q^3$ from $q^3(q-1)$. But the case that $x_{43}=x_{44}=0$ (a total of $q^2$ of these $x_4$) is already ruled out by the restriction that $x_{44}-t_3^{-2}x_{43}\ne 0$. Thus the number of $x_4$'s that make (together with the given (and fixed) independent vectors $x_1=(x_{11}, x_{12},t_3^2x_{14},x_{14})$ and $x_2=(x_{21},x_{22},t_3^2x_{24},x_{24})$) a full rank $A$ is 
\be
q^3(q-1) - q^3 + q^2 =q^2(q-1)^2.
\ee 
Thus if $p\ne 2$, then
\bea\label{e622}
b(\mu,t) &=& (q^3-1)(q^3-q)q^2(q-1)^2\\ \nonumber
            {} &=&q^3(q-1)^4(q+1)(q^2+q+1).\nonumber
\eea
\par
\par\medskip
{\bf (6.2b).} $t_1=1, t_2\ne 1, t_2^2 = t_3^2t_4^2$ or $t_1\ne 1, t_2 = 1, t_1^2 = t_3^2t_4^2$. It turns out that for these cases, $b(\mu, t) = 0$. Consider for example, the case $t_1=1, t_2\ne 1, t_2^2 = t_3^2t_4^2$. These relations are obtained from the first two equations in (\ref{e620}) under the assumption that $t_1 = 1, t_2\ne 1$. Then the third equations in (\ref{e620}) implies that $t_3t_4 = t_2^4$, which in turn implies that $(t_3t_4)^3 = 1$. Then it separates into subcases according to $p = 2$ or not, and $3|(q-1)$ or not. For example, if $p=2$, since $t_2\ne 1 =-1$, $t_3t_4 \ne 1$. It follows that this case happens only if $3|(q-1)$ and $t = (1, t_2, t_3, t_3^{-1}t_2)$, where $t_2$ is a primitive root of $x^3-1$. Thus, the reduced matrix of $D_3$ is
{\small\be
D_3'=\begin{pmatrix} t_3I_4& -t^2\mu^{-1}\\ 0 & R_3 \end{pmatrix}, \;\mbox{where}\; 
R_3 = \mbox{diag}(\ast, 0, \ast, \ast). 
\ee}
This coefficient matrix gives linearly dependent solutions $x_3$ and $x_4$ and thus $b(\mu,t) =0$. Similarly, if $p\ne 2$, both $3|(q-1)$ and $3\notdivides (q-1)$ lead to $b(\mu,t) =0$. 
\par\medskip
{\bf (6.2c).} Both $t_1, t_2 \ne 1$.
\par\medskip
(6.2c1) Assume $3|(q-1)$. $t_1\ne t_2$ are primitive roots of $x^3-1$, $t_3t_4 = t_1$ or $t_3t_4 = t_2$. In both of these cases, $b(\mu, t)=0$. For example, consider $t_3t_4 = t_1$, then we can reduce $D_2$ and $D_3$ to
{\small\be
D_2' = \begin{pmatrix}0 & {} & {} & {}\\ {} & 1-t_2 & {} & {}\\ 
                                {} & {} & t_2 & -t_3^2\\ {} & {} & 0 & t_2^2-t_3^2t_4^2 \end{pmatrix},\;\;
D_3'=\begin{pmatrix} t_3I_4& -t^2\mu^{-1}\\ 0 & R_3 \end{pmatrix}, 
\ee}
where $R_3 = \mbox{diag}(0,\ast,\ast,\ast)$. So the solutions $x_2 = (x_{21},0,0,0)$ and $x_4=(x_{41},0,0,0)$ are dependent. Similarly for $t_3t_4 = t_2$. 
\par\medskip
(6.2c2) Assume $3|(q-1)$. $t_1\ne t_2$ are primitive roots of $x^3-1$ and $t_3t_4 = 1$, there are $2(q-1)$ of these $t$'s. In this case, the solutions look like
\be
x_1=(0,x_{12},0,0), &{}& x_3=(0,0,t_3x_{44}, t_3^{-1}t_4^2x_{43}),\\
x_2=(x_{21},0,0,0), &{}& x_4=(0,0,x_{43},x_{44}).
\ee
Thus we need $t_3x_{44}^2 - t_3^{-1}t_4^2x_{43}^2\ne 0$. Given $t_3, t_4$, and $x_{43}$, we need $x_{44}\ne \pm t_3^{-1}t_4x_{43}$. If $p=2$, there is only one condition $x_{44}\ne t_3^{-1}t_4x_{43}$, so there are $q(q-1)$ of these pairs of $(x_{43},x_{44})$. Since there are $(q-1)^2$ pairs of $x_1, x_2 \ne 0$, we have
\bea\label{e623}
b(\mu,t) = q(q-1)^3. \qquad (3|(q-1), p= 2)
\eea
Similarly,
\bea\label{e624}
b(\mu,t) = (q-1)^4. \qquad (3|(q-1), p\ne 2)
\eea
\par\medskip
(6.2c3) All the following cases lead to $b(\mu,t)=0$: 
\par
(1) $7|(q-1)$, $t_1$ a primitive root of $x^7-1$, $t_1 = t_2^2=t_3^2t_4^2$. 
\par
(2) $t_1 = t_2^2=t_3^2t_4^2$, and $t_3t_4 = t_2^4$. 
\par
(3) $t_1^2 = t_2^2=t_3^2t_4^2$.
\par\medskip
Now multiply (\ref{e621}) and (\ref{e622}) by $q-1$, multiply (\ref{e623}) and (\ref{e624}) by $2(q-1)$, sum them up and divide by $4$, we have the contribution of $\mu = \{1,1,2\}$ to the sum in (\ref{e26}):
{\footnotesize \bea\label{e625}
\qquad B(\{1,1,2\}):=\begin{cases} \frac{1}{4}q(q-1)^4(q^6+q^5+2P_3), & \mbox{if $p=2$},\\
                                     \frac{1}{4}(q-1)^5((q^4+q^3)(q^2+q+1)+2P_3), & \mbox{if $p\ne 2$},
                                     \end{cases}
\eea}
where $P_3$ is the factor indicator as defined by (\ref{e618}).
\par\medskip
%%%%%%%%%
\subsection {The case $\mu=\{1,3\}$}
The matrix in (\ref{e37}) is $C = \mbox{diag}(D_1,D_2)$, where 
{\tiny\be
D_1 =\begin{pmatrix}t_1-t_1^2 & {} & {} & {}\\ {} & t_1 &-t_2^2 & 0 \\{} & 0 & t_1 & -t_3^2\\ {} & -t_4^2 & 0 & t_1\end{pmatrix} &\rightarrow &
D_1' =\begin{pmatrix}t_1-t_1^2 & {} & {} & {}\\ {} & t_1 &-t_2^2 & 0 \\{} & 0 & t_1 & -t_3^2\\ {} &0 & 0& t_1^3-(t_2t_3t_4)^2 \end{pmatrix}\\
D_2 = \begin{pmatrix} t_2I_4& -t^2\mu^{-1} & 0\\ 0 & t_3I_4 & -t^2\mu^{-1} \\
                                  -t^2\mu^{-1} & 0 & t_4I_4 \end{pmatrix} &\rightarrow &
D_2' = \begin{pmatrix} t_2I_4& -t^2\mu^{-1} & 0\\ 0 & t_3I_4 & -t^2\mu^{-1} \\
                                  0 & 0 & t_2t_3t_4I_4-(t^2\mu^{-1})^3 \end{pmatrix}.
\ee}
By eliminating the nonzero factor $t_2t_3t_4$, we have
\bea\label{e631}
{} & & t_2t_3t_4I_4-(t^2\mu^{-1})^3 \rightarrow\\ \nonumber 
& &  \mbox{diag}(t_2t_3t_4-t_1^6, 1-t_2t_3t_4, 1-t_2t_3t_4, 1-t_2t_3t_4).
\eea
Thus the conditions in Theorem 3.1 lead to the following equations (the nonzero factors $t_i$ are omitted):
\bea\label{e632}
(t_1 - 1)(t_1^3-(t_2t_3t_4)^2) & = & 0,\\ \nonumber
(t_2t_3t_4-t_1^6)(t_2t_3t_4-1) &= & 0. 
\eea
These equations in turn lead to the cases listed below for further consideration. For all other choices of $t$, $b(\mu,t) = 0$.
\par\medskip
{\bf (6.3a).} $t_1=t_2t_3t_4 = 1$. There are $(q-1)^2$ of these $t$. Given one such $t$, we obtain the solutions $x_1,x_2,x_3, x_4$ (up to a nonzero scalar multiple) from the corresponding $D_1'$ and $D_2'$, and consider the matrix $(x_1^T, x_4^T, x_3^T, x_2^T)$ ($r_i$ means row $i$; $t_2^2r_3$ means multiply row $3$ by $t_2^2$, etc.):
{\footnotesize\be
\begin{pmatrix}x_{11} & x_{41} & x_{41} & x_{41}\\ 
                       t_2^2t_3^2x_{14} & x_{42} & t_2^2x_{43} & t_2^2t_3^2x_{44}\\ 
                       t_3^2x_{14} & x_{43} & t_3^2x_{44} & t_3^2t_4^2x_{42}\\ 
                       x_{14} & x_{44} & t_4^2x_{42} & t_2^2t_4^2x_{43}\end{pmatrix} 
                       \stackrel{t_2^2r_3, t_2^2t_3^2r_4}{\longrightarrow}
\begin{pmatrix}x_{11} & x_{41} & x_{41} & x_{41}\\ 
                       t_2^2t_3^2x_{14} & x_{42} & t_2^2x_{43} & t_2^2t_3^2x_{44}\\ 
                       t_2^2t_3^2x_{14} & t_2^2x_{43} & t_2^2t_3^2x_{44} & x_{42}\\ 
                       t_2^2t_3^2x_{14} & t_2^2t_3^2x_{44} & x_{42} & t_2^2x_{43}\end{pmatrix}.
\ee}
\par\medskip
(6.3a1). If $x_{11}\ne 0$ ($q-1$ of them), let $b=t_2^2t_3^2x_{14}x_{11}^{-1}x_{41}$, subtract $t_2^2t_3^2x_{14}x_{11}^{-1}r_1$ from $r_2, r_3, r_4$ in the last matrix, we see that we need the following matrix to have rank $3$:
{\small\be
\begin{pmatrix} x_{42} -b& t_2^2x_{43}-b & t_2^2t_3^2x_{44}-b\\ 
                       t_2^2x_{43} -b& t_2^2t_3^2x_{44}-b & x_{42}-b\\ 
                       t_2^2t_3^2x_{44}-b & x_{42}-b & t_2^2x_{43}-b\end{pmatrix}.
\ee}
This matrix is a circulant matrix, its polynomial presenter (in the variable $y$) is
\be
f(y) = (x_{42}-b) + (t_2^2x_{43} - b)y + (t_2^2t_3^2x_{44}-b)y^2.
\ee
As in the case of $\mu = \{3\}$ in (\ref{e58}), we need (c.f. (\ref{e59}))
{\small
\bea\label{e633}
\qquad f(1)=x_{42} + t_2^2x_{43}+ t_2^2t_3^2x_{44}-3b &\ne& 0\;\mbox{and}\\ \nonumber
\mbox{gcd}(f(y), y^2+y+1) = 1. 
\eea}
Use an argument similar to the one in the case $\mu = \{3\}$, notice that $x_{14}$ and $x_{41}$ are arbitrary, we have (under the condition $x_{11}\ne 0$)
\bea\label{e634}
b_1(\mu, t) = \begin{cases} q^4(q-1)^2, &\mbox{if $p =3$},\\
                                         q^2(q-1)^4, &\mbox{if $3|(q-1)$},\\
                                          q^2(q-1)^3(q+1), &\mbox{otherwise}. \end{cases}
\eea
\par\medskip
(6.3a2). Assume $x_{11}=0$. Then we need $x_{14}, x_{41} \ne 0$ and have
{\tiny\be
\begin{pmatrix} 0 & x_{41} & x_{41} & x_{41}\\ 
                       t_2^2t_3^2x_{14} & x_{42} & t_2^2x_{43} & t_2^2t_3^2x_{44}\\ 
                       t_2^2t_3^2x_{14} & t_2^2x_{43} & t_2^2t_3^2x_{44} & x_{42}\\ 
                       t_2^2t_3^2x_{14} & t_2^2t_3^2x_{44} & x_{42} & t_2^2x_{43}\end{pmatrix}
                       \longrightarrow
\begin{pmatrix} 0 & 1 & 1 & 1\\ 
                       1 & x_{42} & t_2^2x_{43} & t_2^2t_3^2x_{44}\\ 
                       1 & t_2^2x_{43} & t_2^2t_3^2x_{44} & x_{42}\\ 
                       1 & t_2^2t_3^2x_{44} & x_{42} & t_2^2x_{43}\end{pmatrix}.
\ee}
\par
If $p=3$, by using row operations, we see that the matrix is not full rank. For example, add both row $2$ and row $3$ to row $4$, we see that the new row $4$ is a multiple of row $1$. Thus $b(\mu,t)=0$ in this case. 
\par
If $p\ne 3$, let 
\be
b' = \frac{1}{3}(1+x_{42}+t_2^2x_{43}+t_2^2t_3^2x_{44}),
\ee
and perform the following operations on the matrix: first add $r_2,r_3,r_4$ to $r_1$ in the last matrix, then subtract $\frac{1}{3}r_1$ from $r_2,r_3,r_4$, to get
{\small\be
\begin{pmatrix} 3 & 3b' & 3b' & 3b'\\ 
                       0 & x_{42}-b' & t_2^2x_{44}-b' & t_2^2t_3^2x_{44}-b'\\ 
                       0 & t_2^2x_{43}-b' & t_2^2t_3^2x_{44}-b' & x_{42}-b'\\ 
                        0& t_2^2t_3^2x_{44}-b' & x_{42}-b' & t_2^2x_{43}-b'\end{pmatrix}.
\ee}
So we are back to a $3\times 3$ circulant matrix similar to the one considered before (case (6.3a1)) with the following difference (c.f. the first condition in (\ref{e633})): we have in this case
\be
x_{42} + t_2^2x_{43}+ t_2^2t_3^2x_{44}-3b' = -1 
\ee
is always $\ne 0$. So by similar arguments, we have (under the assumption that $p\ne 3$, $x_{11}=0$, and $x_{14}, x_{41} \ne 0$):
\bea\label{e635}
b_2(\mu,t) = \begin{cases} (q-1)^5, &\mbox{if $3|(q-1)$};\\
                                       q(q-1)^3(q+1), &\mbox{if $3\notdivides (q-1)$}.
                \end{cases}
\eea
Now we combine (\ref{e634}) and (\ref{e635}) to obtain
{\small \bea\label{e636}
\quad b(\mu, t) = \begin{cases} q^4(q-1)^2, &\mbox{if $p =3$},\\
                                         (q-1)^4(q^2+q-1), &\mbox{if $3|(q-1)$},\\
                                          q(q-1)^3(q+1)^2, &\mbox{otherwise}. \end{cases}
\eea}
\par\medskip
{\bf (6.3b)}. $t_1\ne 1$. All these $t$ lead to $b(\mu,t) = 0$. For example, take the following case from (\ref{e632}): 
\be
t_1^3-(t_2t_3t_4)^2  = 0\quad\mbox{and}\quad
t_2t_3t_4-t_1^6 = 0.
\ee
Then we have $t_1^9 = 1$, and thus $t_1$ is a primitive root of $x^3 - 1$ or $x^9-1$. Let us consider the case that $t_1$ is a primitive root of $x^9-1$ under the assumption that $9|(q-1)$ (for example $q=2^6$). Then the second matrix in (\ref{e631}) is $\mbox{diag}(0, \ast, \ast, \ast)$, where $\ast\ne 0$. This leads to linearly dependent $x_2, x_3, x_4$, so $b(\mu,t) = 0$. We omit the details of the other cases, since the computations are straightforward and similar. 
\par\medskip
Now, multiply (\ref{e636}) by $(q-1)^2$ (the number of $t$) and divide by $3$, we have the contribution of $\mu = \{1,3\}$ to the sum in (\ref{e25}):
{\footnotesize\bea\label{e637}
\qquad\;\; B(\{1,3\}):=\begin{cases} \frac{1}{3}q^4(q-1)^4, &\mbox{if $p=3$};\\
                                    \frac{1}{3}(q-1)^6(q^2+q-1), &\mbox{if $3|(q-1)$};\\
                                    \frac{1}{3}q(q-1)^5(q+1)^2, &\mbox{if $p\ne 3$ and $3\notdivides (q-1)$}.
             \end{cases}
\eea}
\par\medskip
%%%%%%%%%%
\subsection {The case $\mu=\{2,2\}$}
The matrix in (\ref{e37}) is $C = \mbox{diag}(D_1,D_2)$, where 
{\footnotesize\be
D_1 =\begin{pmatrix}t_1I_4 & -t^2\mu^{-1}\\ -t^2\mu^{-1} & t_2I_4 \end{pmatrix} &\rightarrow &
D_1' =\begin{pmatrix}t_1I_4 & -t^2\mu^{-1}\\ 0 & R_1 \end{pmatrix},\\
D_2 =\begin{pmatrix}t_3I_4 & -t^2\mu^{-1}\\ -t^2\mu^{-1} & t_4I_4 \end{pmatrix} &\rightarrow &
D_2' =\begin{pmatrix}t_3I_4 & -t^2\mu^{-1}\\ 0 & R_2 \end{pmatrix},
\ee}
where $t^2\mu^{-1}$ is a diagonal block matrix with diagonal blocks {\small $\begin{pmatrix} 0 & t_1^2\\
t_2^2 & 0\end{pmatrix}$} and {\small $\begin{pmatrix} 0 & t_3^2\\ t_4^2 & 0 \end{pmatrix}$}, and
\bea\label{e641}
R_1 &=& \mbox{diag}(1-t_1t_2, 1-t_1t_2, t_1t_2-t_3^2t_4^2, t_1t_2-t_3^2t_4^2),\\ \nonumber
R_2 &=& \mbox{diag}(t_3t_4-t_1^2t_2^2, t_3t_4-t_1^2t_2^2,1-t_3t_4,1-t_3t_4).
\eea
The conditions in Theorem 3.1 lead to the following system of equations:
\bea\label{e642}
(1-t_1t_2)(t_1t_2-t_3^2t_4^2) & = & 0,\\ \nonumber
(t_3t_4-t_1^2t_2^2)((1-t_3t_4) &= & 0. 
\eea
These equations in turn lead to the cases listed below for further consideration. For all other choices of $t$, $b(\mu,t) = 0$.
\par
We have used $r_i$ for row $i$ of a matrix, and we will use $c_i$ for column $i$ of a matrix. In our matrix operations below, $ac_i$ means multiplying column $i$ by $a$ and $c_i+c_j$ ($i\ne j$) means adding column $j$ to column $i$, in that order.
\par\medskip
{\bf (6.4a).} $t_1t_2=t_3t_4 = 1$. There are $(q-1)^2$ of these $t$. Given one such $t$, we obtain the solutions $x_1,x_2,x_3, x_4$ (up to a nonzero scalar multiple) from the corresponding $D_1'$ and $D_2'$ (note that for these $t$ values, $R_1=R_2=0$):
\be
x_2=(x_{21}, x_{22}, x_{23}, x_{24}),\; x_1=(t_1^2x_{22}, t_2^2x_{21}, t_3^2x_{24}, t_4^2x_{23}),\\
x_4=(x_{41}, x_{42}, x_{43}, x_{44}),\; x_3=(t_1^2x_{42}, t_2^2x_{41}, t_3^2x_{44}, t_4^2x_{43}).
\ee 
Consider the matrix formed by using the $x_i$ as rows ($t_2^{-1} = t_1, t_4^{-1} = t_3$):
{\footnotesize\be
\begin{pmatrix}x_{21} & x_{22} & x_{23} & x_{24}\\ 
                       t_1^2x_{22} & t_2^2x_{21} & t_3^2x_{24} & t_4^2x_{23}\\ 
                       x_{41} & x_{42} & x_{43} & x_{44}\\ 
                       t_1^2x_{42}& t_2^2x_{41}& t_3^2x_{44}& t_4^2x_{43}\end{pmatrix} 
                       \stackrel{t_2^{-2}c_2, \; t_4^{-2}c_4}{\longrightarrow}
\begin{pmatrix}x_{21} & t_1^2x_{22} & x_{23} & t_3^2x_{24}\\ 
                       t_1^2x_{22} & x_{21} & t_3^2x_{24} & x_{23}\\ 
                       x_{41} & t_1^2x_{42} & x_{43} & t_3^2x_{44}\\ 
                       t_1^2x_{42}& x_{41}& t_3^2x_{44}& x_{43}\end{pmatrix}.
\ee}
Let the first row of the last matrix be $(a, b, c, d)$ and let the third row be $(a', b', c', d')$, then the matrix is 
{\footnotesize
\be
\begin{pmatrix} a & b & c & d \\ 
                         b & a & d & c\\ 
                         a' & b' & c' & d' \\ 
                         b' & a' & d' & c'\end{pmatrix} \stackrel{c_2+c_1, c_4+c_3}{\longrightarrow}
\begin{pmatrix} a & a+b & c & c+d \\ 
                         b & a+b & d & c+d\\ 
                         a' & a'+b' & c' & c'+d' \\ 
                         b' & a'+b' & d' & c'+d'\end{pmatrix}.
\ee}
Subtract $r_2$ from $r_1$, subtract $r_4$ from $r_3$, interchange $r_2$ and $r_3$, and interchange $c_2$ and $c_3$, we arrive at 
{\footnotesize
\be
\begin{pmatrix} a-b & c-d & 0 & 0 \\ 
                         a'-b' & c'-d' & 0 & 0\\ 
                         b & d & a+b & c+d \\ 
                         b' & d' & a'+b' & c'+d'\end{pmatrix}.
\ee}
Thus we need the following two matrices to be nonsingular:
{\footnotesize
\bea\label{e643}
\begin{pmatrix} a-b & c-d\\ 
                         a'-b' & c'-d' \end{pmatrix}\quad \mbox{and}\quad
                         \begin{pmatrix} a+b & c+d \\ 
                                            a'+b' & c'+d'\end{pmatrix}.
\eea}
\par
If $p=2$, these two matrices are the same, and the conditions are: $(a+b, c+d)\ne 0$ and $(a'+b',c'+d')$ is not a multiple of $(a+b, c+d)$. The first condition is equivalent to: for a given pair $(a,c)$, $(c,d)\ne (a,c)$. So there are $q^2(q^2-1)$ quadruples $(a,b,c,d)$ that satisfy the condition. Now given such a quadruples $(a,b,c,d)$, to find the number of $(a',b',c',d')$ that satisfy the second condition, we first find the number of $(a',b',c',d')$ such that $(a'+b',c'+d')$ is a multiple of $(a+b, c+d)$.  Given any pair $(a',c')$, a pair $(b',d')$ makes $(a'+b',c'+d')$ a multiple of the given $(a+b, c+d)$ if and only if it is on the line $(a+b, c+d)u+(a',c')$, where $u\in\mathbb{F}_q$ is a parameter. Thus the number is $q$. Then the number of $(a',b',c',d')$ such that $(a'+b',c'+d')$ is not a multiple of $(a+b, c+d)$ is equal to $q^2(q^2-q)$. Thus, the total number of desired $(a,b,c,d)$ and $(a',b',c',d')$ is 
\bea\label{e644}
q^2(q^2-1)\cdot q^2(q^2-q) = q^5(q-1)^2(q+1).\quad (p=2) 
\eea
\par
If $p\ne 2$, the number of $(a,b,c,d)$ such that both $(a+b,c+d)\ne 0$ and $(a-b, c-d)\ne 0$ is $(q^2-1)^2$. This is so since the number of $(a,b,c,d)$ such that either $(a+b,c+d)= 0$ or $(a-b, c-d)= 0$ is $2q^2-1$. To find the number of $(a',b',c',d')$ such that the two matrices in (\ref{e643}) are nonsingular, we follow an approach that is similar to the case $p=2$. The difference is now we have two lines for the pairs $(b',d')$ to avoid. To find the total points in the union of the two lines:
\be
\ell_1 : (a+b, c+d)u-(a',c')\quad\mbox{and}\quad \ell_2 : (a-b, c-d)v+(a',c'),
\ee
where $u,v$ are parameters, we consider their intersection. This leads to the following system:
\be
\begin{pmatrix} a & b \\ c & d\end{pmatrix}\begin{pmatrix} u-v\\u+v\end{pmatrix} = 
\begin{pmatrix} 2a'\\ 2c'\end{pmatrix}.
\ee 
Let $A$ be the coefficient matrix. For a given pair $(a', c')$, a solution of the above system determines $u$ and $v$ uniquely. We separate the discussion according to whether $A$ is singular or not. 
\par
If $A$ is nonsingular (the total number of these $A$ is $(q^2-1)(q^2-q)$), then for any given $(a', c')$, the pair $(u,v)$ is uniquely determined, so $|\ell_1\cup \ell_2|=2q-1$ and the number of $(b',d')$ that are not on either of the two lines is $q^2-2q+1$. The number of $(a', c')$ is $q^2$, so the total number of $(a',b',c',d')$ such that the two matrices in (\ref{e643}) are nonsingular is $q^2(q-1)^2$. Note that if $A$ is nonsingular, then both $(a+b,c+d)\ne 0$ and $(a-b, c-d)\ne 0$. So in this case the total number of $(a,b,c,d)$ and $(a',b',c',d')$ we want is 
\be
(q^2-1)(q^2-q)q^2(q-1)^2 = q^3(q-1)^4(q+1).
\ee
\par
If $A$ is singular, then since $A\ne 0$, its rank is $1$. The number of $(a',c')$ such that $\ell_1$ and $\ell_2$ have nontrivial intersection is $q$ (the system is assumed to be consistent). For these cases, $|\ell_1\cup \ell_2|=q$, and the number of $(a',b',c',d')$ we want is $q(q^2-q)$. If $\ell_1$ and $\ell_2$ do not intersect, then $|\ell_1\cup \ell_2|=2q$, so the number of $(a',b',c',d')$ we want is $(q^2-q)(q^2-2q)$. Thus the total number of $(a',b',c',d')$ we want is 
\be
q(q^2-q)+(q^2-q)(q^2-2q) = q^2(q-1)^2.
\ee
We subtract the number of nonsingular $A$ from the total number of $(a,b,c,d)$ such that both $(a+b,c+d)\ne 0$ and $(a-b, c-d)\ne 0$ to obtain the number of singular ones:
\be
(q^2-1)^2 - (q^2-1)(q^2-q)= (q^2-1)(q-1),
\ee
and then obtain the  total number of $(a,b,c,d)$ and $(a',b',c',d')$ we want in the case that $A$ is singular:
\be
q^2(q-1)^2(q^2-1)(q-1) = q^2(q-1)^4(q+1).
\ee
\par
Finally, the number of $(a,b,c,d)$ and $(a',b',c',d')$ that make the two matrices in (\ref{e643}) nonsingular in the case $p\ne 2$ is: 
\bea\label{e645}
&& q^3(q-1)^4(q+1)+q^2(q-1)^4(q+1)\\
&&\qquad = q^2(q-1)^4(q+1)^2.\quad (p\ne 2) \nonumber
\eea
\par\medskip
Summarize, for $\mu =\{2,2\}$ and each of the $t$ (total $(q-1)^2$) such that $t_1t_2=t_3t_4 = 1$, by (\ref{e644}) and (\ref{e645}) we have
\bea\label{e646}
b(\mu,t)=\begin{cases} q^5(q-1)^2(q+1), & \mbox{if $p=2$};\\
                                  q^2(q-1)^4(q+1)^2, & \mbox{if $p\ne 2$}.
             \end{cases}
\eea
\par\medskip
%%%%%%%%
{\bf (6.4b).} $t_1t_2=t_3^2t_4^2\ne 1$ is a primitive root of $x^3-1$. These cases happen if and only $3|(q-1)$, and there are $2(q-1)^2$ of these $t$ if the condition is satisfied. Given one such $t$, the matrices in (\ref{e641}) are
{\small\be
R_1 = \mbox{diag}(1-t_1t_2, 1-t_1t_2, 0, 0),\; R_2 = \mbox{diag}(0, 0,1-t_3t_4, 1-t_3t_4).
\ee}
we obtain the solutions $x_1,x_2,x_3, x_4$ (up to a nonzero scalar multiple) from the corresponding $D_1'$ and $D_2'$:
\be
x_2=(0, 0, x_{23}, x_{24}),\; x_1=(0, 0, t_3^2x_{24}, t_4^2x_{23}),\\
x_4=(x_{41}, x_{42}, 0, 0),\; x_3=(t_1^2x_{42}, t_2^2x_{41}, 0, 0).
\ee
So we have
{\small\be
\det(x_1^T,x_2^T,x_3^T,x_4^T)\ne 0
                      & \Leftrightarrow& (t_4^2x_{23}^2-t_3^2x_{24}^2)(t_2^2x_{41}^2-t_1^2x_{42}^2)\ne 0 \\
                {}      &\Leftrightarrow& t_4x_{23} \ne \pm t_3x_{24}\; \mbox{and}\; t_2x_{41} \ne \pm t_1x_{42}.
\ee}
From the last two relations, we obtain $b(\mu ,t)$ according to $p = 2$ or not:
\bea\label{e647}
b(\mu,t)=\begin{cases} q^2(q-1)^2, & \mbox{if $p=2$ and $3|(q-1)$};\\
                                     (q-1)^4, & \mbox{if $p\ne 2$ and $3|(q-1)$}.
             \end{cases}
\eea
\par\medskip
Summarize, we have the contribution of $\mu = \{2,2\}$ to the sum in (\ref{e26}) (multiply the corresponding $b(\mu,t)$ by $(q-1)^2$ or $2(q-1)^2$ and use factor indicator $P_3$):
{\footnotesize
\bea\label{e648}
\quad B(\{2,2\}):=\begin{cases} \frac{1}{8}[ q^5(q-1)^4(q+1)+2P_3q^2(q-1)^4], & \mbox{if $p=2$};\\
                                    \frac{1}{8}[q^2(q-1)^6(q+1)^2+2P_3(q-1)^6], & \mbox{if $p\ne 2$}.
             \end{cases}
\eea}
%%%%%%%%%%
\par\medskip
\subsection {The case $\mu=\{4\}$} In this case,
{\small\be
\mu = \begin{pmatrix} 0 & 1\\ I_3 & 0 \end{pmatrix}, \; \mu^{-1} = \begin{pmatrix} 0 & I_3\\ 1 & 0 \end{pmatrix}, \;
t^2\mu^{-1} = \begin{pmatrix} 0 & t_1^2 & 0 & 0 \\ 0 & 0 & t_2^2 & 0\\
                                               0 & 0 & 0 & t_3^2 \\ t_4^2 & 0 &0 & 0 \end{pmatrix}.
\ee}
The matrix in (\ref{e37}) is
{\footnotesize\be
C &=& \begin{pmatrix} t_1I_4 & -t^2\mu^{-1}& 0 & 0 \\ 0 & t_2I_4 & -t^2\mu^{-1} & 0\\
                                               0 & 0 & t_3I_4 & -t^2\mu^{-1} \\ -t^2\mu^{-1} & 0 &0 & t_4I_4 \end{pmatrix}\\
   {} & \rightarrow & \begin{pmatrix} t_1I_4 & -t^2\mu^{-1}& 0 & 0 \\ 0 & t_2I_4 & -t^2\mu^{-1} & 0\\
                                               0 & 0 & t_3I_4 & -t^2\mu^{-1} \\ 0 & 0 &0 & R \end{pmatrix},
\ee}
where $R$ is a $4\times 4$ diagonal matrix with all diagonal entries equal to $t_1t_2t_3t_4 - (t_1t_2t_3t_4)^2$. Thus (\ref{e316}) implies $t_1t_2t_3t_4 = 1$. There are $(q-1)^3$ of these $t$. Given one such $t$, $R = 0$, thus $x_4\ne 0$ is arbitrary, and 
\be
x_3^T = t^2\mu^{-1}x_4^T, x_2^T = (t^2\mu^{-1})^2x_4^T,x_1^T = (t^2\mu^{-1})^3x_4^T.
\ee
Form the matrix with these $x_i$ as rows, multiply $c_2$ by $t_1^2$, multiply $c_3$ by $t_1^2t_2^2$, and multiply $c_4$ by $t_1^2t_2^2t_3^2$, we obtain the following circulant matrix
{\footnotesize\be
A=\begin{pmatrix}x_{41} & t_1^2x_{42} & t_1^2t_2^2x_{43} & t_1^2t_2^2t_3^2x_{44}\\ 
                       t_1^2x_{42} & t_1^2t_2^2x_{43} & t_1^2t_2^2t_3^2x_{44} & x_{41}\\ 
                       t_1^2t_2^2x_{43} & t_1^2t_2^2t_3^2x_{44} & x_{41} & t_1^2x_{42}\\ 
                       t_1^2t_2^2t_3^2x_{44}& x_{41}& t_1^2x_{42}& t_1^2t_2^2x_{43}\end{pmatrix}.
\ee}
Write the first row of $A$ as $(a_0, a_1,a_2,a_3) =: a$, then the polynomial presenter for $A$ is 
\be
f(y) = a_0 +a_1y+a_2y^2+a_3y^3,
\ee
and $A$ is nonsingular if and only if $\mbox{gcd}(f(y),y^4-1)=1$. We discuss the cases according to $p=2$ or not.
\par\medskip
{\bf (6.5a)}. If $p=2$, then $y^4 - 1 = (y-1)^4$, so we need
\be
f(1) \ne 0 \Leftrightarrow a_0+a_1+a_2+a_3 \ne 0.
\ee
Thus there are $q^3(q-1)$ of these $a$ and hence
\bea\label{e651}
b(\mu, t) = q^3(q-1).  \quad (p=2)
\eea 
\par\medskip
{\bf (6.5b)}. If $p\ne 2$, we consider two cases: $4|(q-1)$ or not. 
\par\medskip
(6.5b1).  Assume $4|(q-1)$. Let $\eta$ be a primitive root of $x^4-1$. Then
{\small\be
\mbox{gcd}(f(y),y^4-1)=1 &\Leftrightarrow& f(\eta^k) \ne 0, 0\le k\le 3,\\
&\Leftrightarrow& 
\begin{pmatrix} 1 & 1 & 1 & 1 \\ 1 & \eta & \eta^2 & \eta^3\\
                         1 & \eta^2 & \eta^4 & \eta^6\\ 1 & \eta^3 & \eta^6 & \eta^9
\end{pmatrix}a^T \in (\mathbb{F}^{\times}_q)^4.
\ee}
Since the Vandermonde matrix is nonsingular, there is a bijection between the set of $a$ such that the matrix $A$ is nonsingular and the set $(\mathbb{F}^{\times}_q)^4$, so we have
\bea\label{e652}
b(\mu,t) = (q-1)^4. \quad (p\ne 2, 4|(q-1))
\eea
\par\medskip
(6.5b2). Assume $4\notdivides (q-1)$. Then $y^2+1$ is irreducible, and $y^4-1 = (y^2+1)(y+1)(y-1)$ implies that $\mbox{gcd}(f(y),y^4-1)=1$ if and only if
\bea\label{e653}
{}&{}& a_0+a_1+a_2+a_3\ne 0, \\ \nonumber
{}&{}&\mbox{and}\; a_0-a_1+a_2-a_3\ne 0,\; \\ 
{}&{}&\mbox{and}\;\mbox{gcd}(f(y),y^2+1)=1. \nonumber
\eea
Since $y^2+1$ is irreducible and 
\be
f(y) = (y^2+1)(a_3y+a_2)+(a_1-a_3)y + a_0-a_2,
\ee
$\mbox{gcd}(f(y),y^2+1)=1$ if and only if $a_1-a_3\ne 0$, or $a_1-a_3=0$ but $a_0-a_2\ne 0$. So (\ref{e653}) can be divided into two cases accordingly as follows. 
\par
(6.5b2.1). The conditions on $a$ are
\bea\label{e654}
{}&{}& a_1-a_3 \ne 0, \; \mbox{and}\;a_0+a_1+a_2+a_3\ne 0, \\ \nonumber
{}&{}& \quad\mbox{and}\; a_0-a_1+a_2-a_3\ne 0.
\eea
In this case, if $a_2\ne 0$, then we have the following relation: 
\be
\begin{pmatrix} 0 & 1 & 0 & -1\\ 1 & 1& 1& 1\\ 1 & -1 & 1 & -1\\ 0 & 0 & 1 & 0
\end{pmatrix}a^T\in (\mathbb{F}^{\times}_q)^4,
\ee
which implies that there are $(q-1)^4$ of these $a$ since the coefficient matrix is nonsingular. If $a_2=0$, we have a $3\times 3$ nonsingular matrix instead. So there are $(q-1)^3$ of these $a$. Thus the total number of $a$ that satisfy (\ref{e654}) is:
\be
(q-1)^4+(q-1)^3 = q(q-1)^3.
\ee
\par
(6.5b2.2) The conditions on $a$ are
\be
&& a_1-a_3=0, \;\mbox{and}\; a_0-a_2\ne 0,\\ 
&& \mbox{and}\;a_0+a_1+a_2+a_3\ne 0, \\
&&\mbox{and}\; a_0-a_1+a_2-a_3\ne 0.
\ee
Then $a_1 = a_3$ and
\be
a_0-a_2\ne 0 \; \mbox{and}\;a_0+2a_1+a_2\ne 0 \; \mbox{and}\;a_0-2a_1+a_2\ne 0.
\ee
So similarly, we have $(q-1)^3$ of these $a$. Add the numbers of the cases (6.5b2.1) and (6.5b2.2), we have 
\bea\label{e655}
b(\mu, t) = (q-1)^3(q+1). \quad (p\ne 2, 4\notdivides (q-1))
\eea
\par\medskip
Thus, multiply (\ref{e651}), (\ref{e652}), and (\ref{e655}) by $(q-1)^3$ and divide by $4$, we have the contribution of $\mu = \{4\}$ to the sum in (\ref{e26}):
{\small\bea\label{e656}
\qquad B(\{4\}):=\begin{cases} \frac{1}{4}q^3(q-1)^4, &\mbox{if $p=2$};\\
                                    \frac{1}{4}[P_4(q-1)^7+(1-P_4)(q-1)^6(q+1)], &\mbox{if $p\ne 2$}.
             \end{cases}
\eea}
\par\medskip
Finally, the number of isomorphism classes of $4$-dimensional evolution algebra over $\mathbb{F}_q$ is computed by 
\be
\frac{1}{(q-1)^4}\sum_{i=1}^5B(\{\mu_i\}),
\ee
where $\mu_1 = \{1,1,1,1\}$, $\mu_2=\{1,1,2\}$, $\mu_3=\{1,3\}$, $\mu_4=\{2,2\}$, $\mu_5=\{4\}$, and $B(\{\mu_i\}), 1\le i\le 5$, are given by (\ref{e619}), (\ref{e625}), (\ref{e637}), (\ref{e648}), and (\ref{e656}). To simplify our summary, we define the numbers $b_0, b_i, b_i', 1\le i\le 4$, as follows (recall the factor indication function $P_m$ defined by (\ref{e618})):
{\footnotesize\bea\label{e657}
b_0 &:=& \frac{B(\{\mu_1\})}{(q-1)^4} = \frac{1}{4!}[q^6(q+1)^2(q^2+1)(q^2+q+1)\\ \nonumber
          {} & {} & + P_3(12q(q+1)+6q^2(q+1)^2) + 48(P_7 +P_{15}) + 24P_5]; \nonumber
\eea 
\bea\label{e658}
\qquad\frac{B(\{\mu_2\})}{(q-1)^4} = \begin{cases} \frac{1}{4}q(q^6+q^5+2qP_3)=: b_1, & \mbox{if $p=2$},\\
                                     \frac{1}{4}(q-1)((q^4+q^3)(q^2+q+1)+2P_3)=: b_1', & \mbox{if $p\ne 2$};
                                     \end{cases}
\eea
\bea\label{e659}
\qquad \frac{B(\{\mu_3\})}{(q-1)^4} = \begin{cases} \frac{1}{3}q^4 =: b_2, &\mbox{if $p=3$},\\
                                    \frac{1}{3}[P_3(q-1)^2(q^2+q-1) \\
                                    +(1-P_3)q(q-1)(q+1)^2] =: b_2', &\mbox{if $p\ne 3$};
             \end{cases}
\eea
\bea\label{e6510}
\quad\frac{B(\{\mu_4\})}{(q-1)^4} = \begin{cases} \frac{1}{8}[q^5(q+1)+2P_3q^2]=:b_3, & \mbox{if $p=2$};\\
                                    \frac{1}{8}[q^2(q^2-1)^2+2P_3(q-1)^2]=:b_3', & \mbox{if $p\ne 2$}.
             \end{cases}
\eea
\bea\label{e6511}
\qquad\frac{B(\{\mu_5\})}{(q-1)^4} = \begin{cases} \frac{1}{4}q^3=:b_4, &\mbox{if $p=2$};\\
                                    \frac{1}{4}[P_4(q-1)^3+(1-P_4)(q-1)^2(q+1)]=:b_4', &\mbox{if $p\ne 2$}.
             \end{cases}
\eea
}
\par
\begin{theorem} Notation as above. The number of isomorphism classes $\mathcal{N}(4,\mathbb{F}_q)$ of $4$-dimensional idempotent evolution algebras over a finite field $\mathbb{F}_q$, where $q = p^m$, is given by the formulas in the table below:
{\small\begin{table}[h]
\centering % centering table
\begin{tabular}{c|c} 
%\hline\hline %inserting double-line
%Audio Name&\multicolumn{7}{c}{Sum of Extracted Bits} \\ [0.5ex]
%{} & $k$ \\
%\hline
$p=2$ & $b_0+b_1+b_2'+b_3+b_4$\\
\hline
{} & {} \\
$p=3$ & $b_0+b_1'+b_2+b_3'+b_4'$\\
\hline
{} & {} \\
$p>3$ & $b_0+b_1'+b_2'+b_3'+b_4'$\\
\hline
\end{tabular}
%\hline % inserts single-line
\label{t5}
\end{table}}
\end{theorem}
\par\medskip
\begin{example} By Theorem 6.1, the number of isomorphism classes of $4$-dimensional idempotent evolution algebras over $\mathbb{F}_2$ is  
\be
\mathcal{N}(4,\mathbb{F}_2) &=& \frac{1}{4!}(2^6(2+1)^2(2^2+1)(2^2+2+1))+\frac{1}{4}2(2^6+2^5)\\
                                               && + \frac{1}{3}(2(2-1)(2+1)^2)+\frac{1}{8}2^5(2+1)+ \frac{1}{4}2^3\\
                                               &=& 908,
\ee
and the number of isomorphism classes of $4$-dimensional idempotent evolution algebras over $\mathbb{F}_5$ is  
\be
\mathcal{N}(4,\mathbb{F}_5) &=& \frac{1}{4!}[5^6(5+1)^2(5^2+1)(5^2+5+1) + 24]\\
                                               && +  \frac{1}{4}(5-1)(5^4+5^3)(5^2+5+1)\\
                                               && + \frac{1}{3}5(5-1)(5+1)^2 + \frac{1}{8}5^2(5^2-1)^2 + \frac{1}{4}(5-1)^2(5+1)\\
                                               &=& 18,915,940.
\ee
\end{example}
\par\medskip
%%%%%%%%%%%%%%%%%%
\section*{Acknowledgements} Wei acknowledges the support of National Natural Science Foundation of China (11961050) and the Guangxi Natural Science Foundation (2020GXNSFAA159053). Zou acknowledges the support of a Simons Foundation Collaboration Grant for Mathematicians (416937).
\par
The authors would like to thank the anonymous referee for valuable comments that improved the quality of this paper.
%%%%%%%%%%%%%%%%%%%%%%%%%%%%%%%%%%%%%%%%%%%%%%%%%%%%%%%%%%%%%%%%%%%%%%%%%%
\par\medskip
%\begin{references}
 
%\end{references}

\end{document}